\documentclass{svjour3}
\usepackage[english]{babel}
\usepackage[margin=3cm]{geometry}
\usepackage[]{algorithm2e,subfig}
\usepackage{amsmath,amssymb,mathtools} 
\usepackage{cite,enumerate,float,pgfplots}
\usepackage{courier,dsfont}
\usepackage[utf8]{inputenc}

\title{Error estimation for the time to a threshold value in evolutionary partial differential equations}

\author{ Jehanzeb H. Chaudhry \and Donald Estep \and Trevor Giannini \and Zachary Stevens \and Simon J. Tavener}
\date{\today}
\institute{J. H. Chaudhry \and T. Giannini \and Z. Stevens \at  The University of New Mexico, Albuquerque, NM 87131, USA. \email{jehanzeb@unm.edu, tgianni@math.unm.edu, zstevens@unm.edu}
\and
S. J. Tavener \at Colorado State University,  Fort Collins, CO 80523, USA. \email{tavener@math.colostate.edu}
\and
D. Estep \at Simon Fraser University, Burnaby, BC V5A 1S6, Canada. \email{donald\_estep@sfu.ca}
}

\begin{document}
\maketitle


\begin{abstract}

We develop an \textit{a posteriori} error analysis for a numerical estimate of the time at which a functional of the solution to a partial differential equation (PDE) first achieves a threshold value on a given time interval. This  quantity of interest (QoI) differs from classical QoIs which are modeled as bounded linear (or nonlinear) functionals {of the solution}. Taylor's theorem and an adjoint-based \textit{a posteriori} analysis is used to derive computable and accurate error estimates in the case of semi-linear parabolic and hyperbolic PDEs. The accuracy of the error estimates is demonstrated through numerical solutions of the one-dimensional heat equation and linearized shallow water equations (SWE), representing parabolic and hyperbolic cases, respectively.
\end{abstract}



\section{Introduction}
\label{sec:Introduction}

Numerical solutions of parabolic and hyperbolic differential equations are essential
to the study of physical phenomena that evolve over space and time. To meet the requirements of uncertainty quantification, adjoint-based methods for estimating the error in a quantity of interest that can be expressed as a linear or nonlinear functional of the solution are well developed. However, the \emph{time} at which a functional of the solution achieves a threshold value, which we refer to as the time to an ``event'', is often the primary problem for a computational study of a physical system. Examples include: the time at which the concentration of a chemical species reaches a critical value at a particular location,  the time at which the temperature at a particular location drops below a critical value, and the time at which a traveling wave reaches a certain location given specific initial values and topography. Unfortunately, the error in an estimated time to a particular event cannot be quantified using standard approaches.

Consider a semi-linear evolutionary PDE of the form
\begin{equation} \label{eq:semilinear_evo_PDE}
u_t(x,t) + L u(x,t) = f(u,x,t), \quad x \in \Omega, \; t \in (0,T],
\end{equation}
with appropriate initial and boundary conditions. Here $u_t=\frac{\partial u}{\partial t}$, $L$ is a differential operator in $x$ that is linear in $u$, and $f(u,x,t)$ is a differentiable function.
While $u, L$ and $f$ may be scalar or vector valued, these cases are not distinguished in our analyses since this will be obvious from the context.

Classical \textit{a posteriori} error analysis considers quantities of interest (QoIs) that can be expressed as bounded functionals of the solution and  utilizes variational analysis, generalized Green's functions and computable residuals \cite{Est95, AinOde00, EstLarWil00, BecRan01, GilSul02, BanRan03, Bar04, EstHolMik02, CaoPet04, houston2017adjoint}. Recent extensions to multiscale and multiphysics systems include \cite{CarEstTav08, Est09, EstGinTav12, JohChaCarEstGinLarTav15, Cha18}.  Finite difference and volume methods, and explicit time integration schemes may be analyzed by reformulating these schemes as equivalent finite element methods, see e.g., \cite{ColEstTav14, ChaEstGinShaTav15, ChaEstGinTav15, ColEstTav15, ChaEstTavCarSan16, ChaColSha17, ChaShaWil19}.  Nonlinear QoIs are typically handled by linearization around a computed solution, e.g., \cite{BecRan01}. 
The key point motivating this paper is that the  \emph{time} a functional of the solution of a PDE first achieves a threshold value cannot be expressed as a bounded linear functional of $u$, nor can it be trivially linearized.

Let $G(u;t)$ be a linear functional of $u(x,t)$ which is implicitly dependent on $t$ through $u$, and let $R$ be a threshold value of $G(u;t)$. Assume there are one or more times $t^\star$ during the interval $(0,T]$ for which $G(u;t^\star) = R$ is satisfied and define the time $H(u,\tau)$ as
\begin{equation} \label{eq:defn_H}
H(u,\tau) = \min_{t \in (\tau, T]} \underset{t}\arg ( G(u; t) = R ).
\end{equation}
Here $\tau$ is specified in order to obtain different occurrences of the event $G(u; t)=R$. For example, $H(u,0)$ is the  time  on the interval (0,T] of the \emph{first} occurrence of the event of the $G(u; t) = R$. Similarly, the time the \emph{second} occurrence of the event $G(u; t) = R$ is $H(u,\tau)$, where $H(u,0) \leq \tau < H(u,H(u,0))$, i.e., for any $\tau$ between the times of first and second occurrence of the event. The quantity of interest $Q(u)$ for a fixed $\tau$ is defined as
\begin{equation} \label{eq:QoI_time_to_event}
Q(u) = H(u,\tau).
\end{equation}

In recent work \cite{ChaEstSteTav21}, an \textit{a posteriori} error analysis was performed in order to estimate the error in the time to an event (the time of the first occurrence of a functional achieving a threshold value), in the context of ordinary differential equations (ODEs). In the current work, the analysis of \cite{ChaEstSteTav21} is extended to evolutionary PDEs and to the QoI specified by \eqref{eq:QoI_time_to_event}. An error estimate is derived using Taylor series linearized around the approximate (computed) value of the QoI. The error estimate, which contains terms involving the unknown error of the solution to the PDE, is made computable through the use of solutions to certain adjoint problems. In comparison with the earlier work of \cite{ChaEstSteTav21}, the current analysis requires the solution of an additional adjoint problem arising due to the presence of the spatial operator $L$. We apply the analysis to the one-dimensional heat equation and to the one-dimensional linearized shallow water equations (SWE).

Numerical schemes to solve the semi-linear evolutionary PDE \eqref{eq:semilinear_evo_PDE} are developed in \S \ref{sec:Numerical_Methods}. A variational formulation is introduced in \S \ref{sec:Variational_Formualation} and a space-time Galerkin finite element discretization in \S \ref{sec:Galerkin_Method}. The error in QoI \eqref{eq:QoI_time_to_event} is defined in \S \ref{sec:Error_Estimate} and the  \emph{a posteriori} error analysis for this QoI is presented in \S \ref{sec:Error_Estimate_Threshold}. A computable estimate is developed in \S\ref{sec:Error_Estimate_Computable} by taking advantage of classical error estimates which are reviewed in \S \ref{sec:Error_Estimate_Classic}. Error estimates for two model problem are developed in \S \ref{sec:Model_Problems}, the one-dimensional heat equation in \S \ref{sec:Heat_Equation} and the one-dimensional linearized shallow water equations  in \S \ref{sec:SWE}, respectively. Numerical results for both the linear and nonlinear heat equations and for the linearized shallow water equations with a range of different bottom topographies are presented in \S \ref{sec:Numerical_Results}.


\section{Numerical methods}
\label{sec:Numerical_Methods}

Adjoint-based, \textit{a posteriori} error estimates utilize the variational (or weak) form of PDEs. A variational formulation for the semi-linear evolutionary PDE \eqref{eq:semilinear_evo_PDE} is introduced in \S \ref{sec:Variational_Formualation} and a space-time Galerkin finite element discretization is presented in \S \ref{sec:Galerkin_Method}.

\subsection{Variational formulation}
\label{sec:Variational_Formualation}
Let $W$ be a generic Hilbert space over the domain $\Omega$ with inner product  $\left(\cdot,\cdot \right)_W$ and the induced W-norm $\| \cdot \|_W$.  Let $L^2(0,T;W)$ denote the space of functions $w$ whose $W$-norm is square integrable $\forall t \in [0,T]$. The variational form of \eqref{eq:semilinear_evo_PDE} is: Find $u \in L^2(0,T;W)$ such that
\begin{equation} \label{eq:time_evo_weak}
\left( u_t(\cdot,t), v(\cdot,t) \right) + \left( L_1 u(\cdot,t), L_2 v(\cdot,t) \right)
  = \left( f(u,\cdot,t), v(\cdot,t) \right), \quad \forall v \in W, \quad \forall t \in (0,T],
\end{equation}
where $f(u,\cdot,t)\in W$ for any $t\in(0,T]$ and $\left( a, b \right) = \int_{\Omega} a\cdot b \; {\rm d}\Omega$ is the spatial $L^2$ inner product over $\Omega$. The two operators $L_1,L_2$ are linear differential operators such that
\begin{equation} \label{eq:operator_split}
L_2^* L_1 u = L u.
\end{equation}
The operator $L_2^*$ is the formal adjoint of $L_2$ which satisfies the property that
\begin{equation} \label{eq:formal_adjoint}
    \left( u, L_2 v \right) = \left( L_2^* u, v \right), \qquad \forall u \in W, v \in W.
\end{equation}
Under suitable conditions on the operators $L_1$ and $L_2$ (and hence $L$), \eqref{eq:time_evo_weak} is well-posed; the reader is referred to \cite{evans10} for more details. Although the differential operator $L_1$ may be of lower order than $L$, if a solution $u$ to \eqref{eq:time_evo_weak} is sufficiently differentiable, it also satisfies \eqref{eq:semilinear_evo_PDE}.


\subsection{Space-time Galerkin finite element method}
\label{sec:Galerkin_Method}

Let $\mathcal{T}_h$ be a shape-regular triangulation of the spatial domain $\Omega$ where $h$ denotes the diameter of the largest element. For the numerical examples in \S \ref{sec:Numerical_Results}, $\Omega$ is one-dimensional and $\mathcal{T}_h$ is a uniform set of $N_x$ subintervals of the domain where $h$ denotes the width of the subintervals. Let $V^{q_s}$ be the standard space of continuous piecewise polynomials of degree $q_s$ on $\mathcal{T}_h$. 
This space may be scalar or vector valued depending upon the context.

The time domain $[0, T]$ is discretized into $N_t$ uniform subintervals of width $p$ denoted $\mathcal{T}_p$, i.e.,  $\mathcal{T}_p = \{I_0, \dots, I_n, \dots, I_{N_t-1}\}$ where $I_n = [t_n, t_{n+1}]$. Let $l_j(t)$ be the $j$th degree Lagrange polynomial in time on the interval $I_n$. For the space-time slab $S_n = \Omega \times I_n$ the solution space is defined as 
\begin{equation} \label{eq:FE_space}
W_{n}^{q_t, q_s} = \left \{ w(x,t) \, | \, w(x,t) = {\sum_{j = 0}^{q_t} l_j(t) v_j(x), \; v_j(x)} \in V^{q_s}, \, (x,t)\in S_n \right \}.
\end{equation}
%
We define the continuous Galerkin method, cG($q_t,q_s$) as: Find a continuous in time $U \in W_{n}^{q_t,q_s}$ {for $n=0,\dots,N_t-1$,} such that 
\begin{equation} \label{eq:cG_forward}
  \int_{t_n}^{t_{n+1}}\Big[ \Big( U_t(\cdot, t), v(\cdot, t )\Big) + \Big( L_1 U(\cdot, t), L_2 v(\cdot, t) \Big) \Big] \; {\rm d}t
= \int_{t_n}^{t_{n+1}} \Big( f(\cdot, t), v(\cdot, t) \Big) \; {\rm d}t, \hspace{0.5cm} \forall  v \in W_{n}^{q_t-1, q_s}.
\end{equation}



\section{\emph{A posteriori} error estimation}
\label{sec:Error_Estimate}

Let $u(x,t)$ be the true solution of \eqref{eq:semilinear_evo_PDE} and $G(u; t)$ be a time-dependent linear functional of $u(x,t)$ that may be expressed in terms of the inner product,
\begin{equation} \label{eq:G}
G(u; t) = \int_{\Omega} w(x)\cdot u(x,t) \; {\rm d}x =\left( w(\cdot), u(\cdot,t) \right),
\end{equation}
where $w(x)$ is some differentiable weight function. Assume that $\left.w(x)\right|_{\partial\Omega} =0$. For a chosen threshold $R$ and fixed value of $\tau$ in \eqref{eq:defn_H}, denote the true value of the QoI
\begin{equation} \label{def:t_t}
  t_t:=Q(u).
\end{equation}
Let $U(x,t)$ be an approximation of the solution satisfying \eqref{eq:cG_forward} and
\begin{equation} \label{def:t_c}
t_c := Q(U),
\end{equation}
the corresponding computed value of the QoI.  Figure \ref{fig:Reference_QoI_SWE_Constant_ThirdEvent_CloseUp} illustrates these definitions. We denote the error in the approximate \emph{solution} as $e(x,t) = u(x,t) - U(x,t)$. As noted in \S \ref{sec:Introduction}, the \textit{a posteriori} error analysis for functionals of $e(x,t)$ is well established. However, we seek to  estimate the error
\begin{equation} \label{eq:error}
e_Q = t_t-t_c,
\end{equation}
and specifically, to calculate a \emph{computable} estimate $\nu \approx e_Q$. We define the \emph{effectivity ratio} of this estimate  to be
\begin{equation} \label{eq:effectivity}
\rho_{\rm eff}=\frac{\nu}{e_Q},
\end{equation}
where clearly the goal is to construct an error estimate for which $\rho_{\rm eff} \approx 1$. We first derive an error representation in \S \ref{sec:Error_Estimate_Threshold}. Standard \emph{a posteriori techniques} (see \S \ref{sec:Error_Estimate_Classic}) are then employed to approximate the error in three linear functionals of $e(x,t)$. The errors in these three functionals are calculated as the inner product of solutions to three distinct adjoint problems and computable residuals as described in \S \ref{sec:Error_Estimate_Computable}.

\begin{figure}[ht]
\centering
\subfloat[ {$G(u;t)$ for $t \in [0,200]$} ]{
\includegraphics[width=8cm]{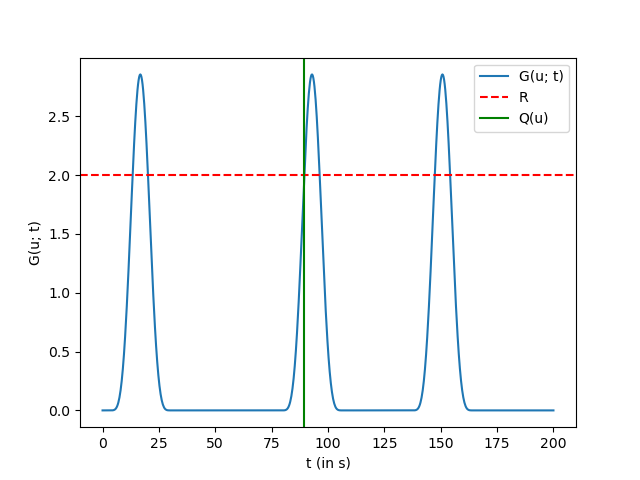} }
\hfill
\subfloat[ {$G(u;t)$ and $G(U;t)$ for $t \in [89,91]$} ]{
\includegraphics[width=8cm]{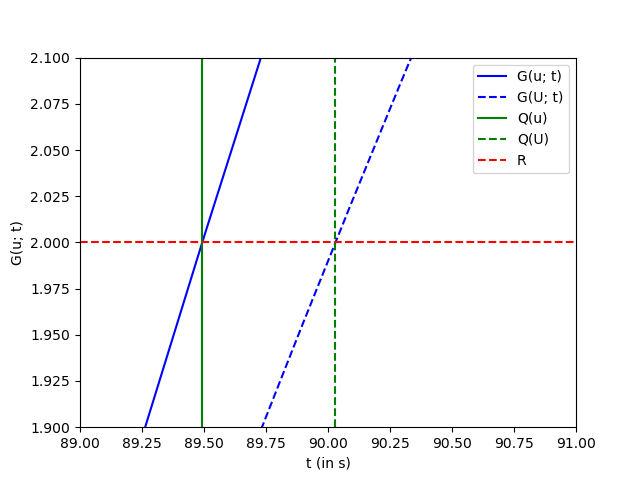} }
\caption{(a) Reference values of the functional $G(u;t)$ for \S \ref{sec:SWE_Constant}. The QoI is the time for the third occurrence of the event $G(u;t)=R$. (b) Detail near the time of third occurrence of the event $G(u;t)=R$ showing the true and approximate values of the functional, $G(u;t)$ and $G(U;t)$ respectively, and the corresponding true and approximate values of the QoI, $Q(u)$ and $Q(U)$. The error  is $e_Q = Q(u)-Q(U)$.}
\label{fig:Reference_QoI_SWE_Constant_ThirdEvent_CloseUp}
\end{figure}

\subsection{Error estimation for the time to an event $G(u;t)=R$}
\label{sec:Error_Estimate_Threshold}

An adjoint-based \emph{a posteriori} error estimate for $e_Q$ as defined in \eqref{eq:error} is presented in Theorem \ref{thm:NS_error_rep} below.

\begin{theorem}\label{thm:NS_error_rep}
Consider the QoI defined by \eqref{eq:QoI_time_to_event}, where the functional $G$ is defined by \eqref{eq:G}, and $t_t$, $t_c$, and the error in the QoI, $e_Q$, are defined by \eqref{def:t_t}, \eqref{def:t_c} and \eqref{eq:error} respectively.  
Let 
\begin{equation} 
\label{eq:denom_err_rep}
\mathcal{D}(U,t_c) = \left(L_2w(\cdot), L_1U(\cdot, t_c) \right) - \left(w(\cdot),f(U,\cdot, t_c) \right) - \left(\nabla_uf^\top(U,\cdot,t_c)w(\cdot),e(\cdot,t_c) \right) + \left(L_2w(\cdot), L_1e(\cdot,t_c)\right),
\end{equation}
where
$e(x,t) = u(x,t) - U(x,t)$, and  {$\nabla_uf$ is a square matrix with entries $[\nabla_uf]_{ij} = \partial f_i / \partial u_j$}. For semi-linear time evolution PDEs \eqref{eq:semilinear_evo_PDE}, the error in the QoI is given by
\begin{equation}\label{eq:err_rep}
\begin{aligned}
&e_Q = t_t-t_c  =  \frac{\left(w(\cdot), e(\cdot, t_c) \right) + \left( w(\cdot), \mathcal{R}_1(\cdot,t_c,t_t) \right)}{ \mathcal{D}(U,t_c)  - \left(w(\cdot), \mathcal{R}_2(u,U,\cdot,t_c)  \right) }, 
\end{aligned}
\end{equation}
where the two remainder terms are $\mathcal{R}_1(x,t_c,t_t) = \mathcal{O}\left((t_t-t_c)^2\right)$ and $\mathcal{R}_2(u,U,x,t_c) = \mathcal{O}\left((u-U)^2\right)$.
\end{theorem}

\begin{proof}
Using the definitions of $t_t$ and $t_c$,
\begin{equation} \label{eq:def_tt_tc}
G(u(t_t)) - G(U(t_c)) = R - R = 0.
\end{equation}
From \eqref{eq:def_tt_tc} and \eqref{eq:G}
\begin{equation*}
\left( w(\cdot), u(\cdot,t_t) - U(\cdot,t_c)  \right) = 0.
\end{equation*}
Expanding the functional $G(u(t_t))$ about $t_c$ by expanding $u(x, t_t)$ around $t_c$ using Taylor Series,
\begin{equation} \label{eq:rearrange_this}
\begin{aligned}
0
&= \left(w(\cdot), u(\cdot, t_c) + (t_t - t_c)u_t(\cdot, t_c) + \mathcal{R}_1(\cdot, t_c, t_t)  - U(\cdot, t_c) \right)  \\
&= \left(w(\cdot), e(\cdot, t_c) \right) + (t_t - t_c) \left(w(\cdot) , u_t(\cdot, t_c) \right) + \left( w(\cdot), \mathcal{R}_1(\cdot, t_c, t_t) \right),
\end{aligned}
\end{equation}
where $e(x,t) = u(x,t) - U(x,t)$ and the remainder $\mathcal{R}_1(x, t_c, t_t)$ is $\mathcal{O}\left((t_t-t_c)^2\right)$. Rearranging \eqref{eq:rearrange_this},
\begin{equation} \label{eq:err_rep1}
t_t-t_c  =  - \frac{\left(w(\cdot), e(\cdot, t_c) \right) + \left( w(\cdot), \mathcal{R}_1(\cdot,t_c,t_t) \right)}{\left(w(\cdot), u_t(\cdot,t_c)\right)}.
\end{equation}
From the weak formulation of the PDE \eqref{eq:time_evo_weak},
\begin{equation} \label{eq:replace_ut}
\begin{aligned}
\left(w(\cdot), u_t(\cdot,t_c)\right)
&= -\left( L_2 w(\cdot), L_1 u(\cdot,t_c)\right) + \left( w(\cdot), f(u,\cdot,t_c) \right)  \\
&= -\left( L_2 w(\cdot), L_1 U(\cdot,t_c)\right) - \left( L_2 w(\cdot), L_1 e(\cdot,t_c)\right)
     + \left( w(\cdot), f(u,\cdot,t_c) \right).
\end{aligned}
\end{equation}

Using Taylor's Theorem once more to expand $f(u,x,t_c)$ around $U(x,t_c)$,
\begin{align}
    \left(w(\cdot), f(u,\cdot,t_c)\right)
    &= \left(w(\cdot), f(U,\cdot,t_c)\right) + \left(\nabla_uf^\top(U,\cdot,t_c)w(\cdot),e(\cdot,t_c)\right) + \left(w(\cdot), \mathcal{R}_2(u,U,\cdot,t_c) \right),
    \label{eq:f_expand2}
\end{align}
where the remainder $\mathcal{R}_2(u,U,x,t_c)$ is of order $\mathcal{O}\left((u-U)^2\right)$. Substituting \eqref{eq:f_expand2} into \eqref{eq:replace_ut} and then into \eqref{eq:err_rep1} yields the final result.

\end{proof}

\begin{corollary}\label{cor:lin_err}
If the differential equation \eqref{eq:semilinear_evo_PDE} is linear, so that $f=f(x,t)$ does not depend on $u$, then the error representation is
\begin{equation}\label{eq:err_rep_lin}
t_t-t_c  =
\frac{\left(w(\cdot), e(\cdot, t_c) \right) + \left( w(\cdot), \mathcal{R}_1(\cdot,t_c,t_t) \right)}
     {\left(L_2 w(\cdot), L_1 U(\cdot, t_c) \right) - \left( w(\cdot), f(\cdot, t_c) \right)
        +  \left(L_2 w(\cdot), L_1 e(\cdot,t_c)\right) }.
 \end{equation}
\end{corollary}

\subsection{Error estimation for a linear functional}
\label{sec:Error_Estimate_Classic}

In order to define a computable error estimate, standard adjoint-based \emph{a posteriori} error analysis is employed to approximate the three terms in \eqref{eq:err_rep} that are linear functionals of $e(x,t)$. The standard analysis is presented for completeness in Theorem \ref{thm:standard_error_rep} below.

\begin{theorem}\label{thm:standard_error_rep}
Given a numerical solution $U(x,t)$ to \eqref{eq:semilinear_evo_PDE} and data $\psi(x)$, for any $\hat{t} \in (0,T]$ the error $\left(\psi(\cdot),e(\cdot,\hat{t})\right)$ is given by
\begin{equation}\label{eq:standard_error_rep}
        \left(\psi(\cdot),e(\cdot,\hat{t})\right) = \left(\phi(\cdot, 0), e(\cdot, 0)\right) + \int_{0}^{\hat{t}}
        \left(\phi(\cdot,t),f(U,\cdot,{t})-U_t(\cdot,{t}) \right)\; {\rm d}t - \int_{0}^{\hat{t}}\left(L_2\phi(\cdot,{t}),L_1U(\cdot,{t}) \right)\; {\rm d}t
\end{equation}
where $\phi(x,t)$ is the solution of the adjoint problem
\begin{equation} \label{eq:strong_adj}
\begin{aligned}
-\phi_t(x, t) { + L^* \phi(x, t)} & =  A_{u,U}^* \phi(x,t), \quad &&x \in \Omega, \quad         &t \in [0,\hat{t}),  \\
 \phi(x, t)   & = 0, \qquad\qquad\;     &&x \in \partial\Omega, \quad  &t \in [0,\hat{t}),  \\
 \phi(x, \hat{t})  & = \psi(x), \qquad\quad  &&x \in \Omega. \\
\end{aligned}
\end{equation}
The operator $A_{u,U}^*$ is the adjoint of the linear operator
\begin{equation}
A_{u,U} = \int_0^1 \frac{df}{dz}(z,x,t) \; {\rm d}s,
\end{equation}
where $z=su+(1-s)U$.
\end{theorem}

\begin{proof}
Multiplying the adjoint equation \eqref{eq:strong_adj} by the error $e(x,t) = u(x,t)-U(x,t)$, and integrating over the space-time domain $\Omega \times [0,\hat{t}]$,
\begin{equation} \label{eq:int_of_adj}
0=\int_{0}^{\hat{t}} \left( \phi_t(\cdot, t), e(\cdot,t) \right)  \; {\rm d}t
  -\int_{0}^{\hat{t}} \left( L^* \phi(\cdot, t), e(\cdot,t) \right)  \; {\rm d}t +\int_{0}^{\hat{t}} \left( A_{u,U}^* \phi(\cdot,t), e(\cdot,t) \right)\; {\rm d}t.
\end{equation}
Integrating the first term of \eqref{eq:int_of_adj} by parts in time and enforcing the initial condition $\phi(x,\hat{t})=\psi(x)$, 
\begin{equation*}
    \int_{0}^{\hat{t}} \left( \phi_t(\cdot, t), e(\cdot,t) \right)  \; {\rm d}t 
    = {\left(\psi(\cdot) , e(\cdot, \hat{t}) \right)
      - \left(\phi(\cdot ,  0) , e(\cdot, 0) \right)
      - \int_{0}^{\hat{t}} \left(\phi(\cdot , t) , e_t(\cdot,t) \right)\; {\rm d}t. }
\end{equation*}
From \eqref{eq:operator_split}, $L^*=L_1^*L_2$, and the second term of \eqref{eq:int_of_adj} becomes
\begin{equation} \label{eq:L_adj_decomp}
   \int_{0}^{\hat{t}} \left( L^*\phi(\cdot,t), e(\cdot,t) \right)\; {\rm d}t
 = \int_{0}^{\hat{t}} \left( L_1^*L_2 \phi(\cdot,t), e(\cdot,t) \right)\; {\rm d}t
 = \int_{0}^{\hat{t}} \left( L_2\phi(\cdot,t), L_1e(\cdot,t) \right)\; {\rm d}t.
\end{equation}

The operator $A_{u,U}$ has the property
\begin{equation}\label{eq:A_property}
    A_{u,U}e(x,t) = f(u,x,t) - f(U,x,t),
\end{equation}
hence, using the property of the adjoint \eqref{eq:formal_adjoint} and \eqref{eq:A_property}, the third term of \eqref{eq:int_of_adj} becomes
\begin{equation} \label{eq:adj_op_err_property}
   \int_{0}^{\hat{t}} \left( A_{u,U}^* \phi(\cdot,t), e(\cdot,t) \right)\; {\rm d}t
 = \int_{0}^{\hat{t}} \left(\phi(\cdot,t), A_{u,U}e(\cdot,t) \right)\; {\rm d}t
 = \int_{0}^{\hat{t}} \left(\phi(\cdot,t), f(u,\cdot,t) - f(U,\cdot,t) \right)\; {\rm d}t.
\end{equation}
Combining \eqref{eq:int_of_adj}, \eqref{eq:L_adj_decomp}, and \eqref{eq:adj_op_err_property} yields,
\begin{equation*}
\begin{aligned}
0&= {
    \left(\psi(\cdot) , e(\cdot, \hat{t}) \right)
    - \left(\phi(\cdot ,  0) , e(\cdot, 0) \right)
    -\int_{0}^{\hat{t}} \left(\phi(\cdot , t) , e_t(\cdot,t) \right)\; {\rm d}t}, \\
    &\hspace{1cm} - \int_{0}^{\hat{t}} \left( L_2\phi(\cdot,t), L_1e(\cdot,t) \right)\; {\rm d}t + \int_{0}^{\hat{t}} \left(\phi(\cdot,t), f(u,\cdot,t) - f(U,\cdot,t) \right)\; {\rm d}t.
\end{aligned}
\end{equation*}
Since $e(x,t) = u(x,t)-U(x,t)$
\begin{equation*}
\begin{aligned}
&\left(\psi(\cdot) , e(\cdot,\hat{t}) \right) \\
&=\left( \phi(\cdot, 0), e(\cdot, 0) \right)
    + \int_{0}^{\hat{t}} \left( \phi(\cdot, t) , u_t(\cdot, t) - U_t(\cdot, t) \right) \; {\rm d}t \\
    & \hspace{1cm} {+} \int_{0}^{\hat{t}} \left( L_2 \phi(\cdot, t), L_1 u(\cdot ,t) \, {-} \, L_1 U(\cdot, t) \right) \; {\rm d}t
    \;{-}\int_{0}^{\hat{t}} \left(\phi(\cdot,t), f(u,\cdot,t) - f(U,\cdot,t) \right)\; {\rm d}t.
\end{aligned}
\end{equation*}
Noting that the exact solution $u(x,t)$ satisfies \eqref{eq:time_evo_weak} produces the desired result.

\end{proof}

In practice, since operator $A_{u,U}$ requires the true solution to \eqref{eq:semilinear_evo_PDE}, it is approximated by
\begin{equation*}
    A_{u,U} \approx \nabla_uf(U,x,t).
\end{equation*}
With this, the right side of the adjoint equation \eqref{eq:strong_adj} is approximated by
\begin{equation*}
    A_{u,U}^*\phi(x,t) \approx \left( \nabla_uf(U,x,t) \right)^{\top} \phi(x,t).
\end{equation*}

\begin{remark}
Theorem \ref{thm:standard_error_rep} is formulated in terms of the strong form of the adjoint problem \eqref{eq:strong_adj}, but clearly it is adequate for $\phi$ to satisfy the weak form of the problem
\begin{equation} \label{eq:adj_weak}
-\left( \phi(\cdot,t), v(\cdot) \right) + \left( L_2 \phi(\cdot,t), L_1 v(\cdot) \right) = \left( \phi(\cdot,t), A_{u,U}v(\cdot)  \right) \qquad \forall v \in H_0^1(\Omega),\forall t \in (0,\hat{t}].
\end{equation}
\end{remark}

\subsection{A computable error estimate for the time to an event $G(u;t)=R$}
\label{sec:Error_Estimate_Computable}

The magnitudes of the remainders $\mathcal{R}_1$ and $\mathcal{R}_2$ in the error representation \eqref{eq:err_rep} decrease more rapidly than other terms as the accuracy of the numerical solution increases (e.g., by increasing $N_x,N_t,q_t$ or $q_s$ in \S~\ref{sec:Galerkin_Method}). Setting both $\left( w(\cdot), \mathcal{R}_1(\cdot,t_c,t_t) \right) \approx 0$ and $\left( w(\cdot), \mathcal{R}_2(u,U,\cdot,t_c) \right) \approx 0$ gives a first approximation

\begin{equation}\label{eq:error_estimate_computable}
\begin{aligned}
t_t-t_c  &\approx  \frac{\left(w(\cdot), e(\cdot, t_c) \right)}{\mathcal{D}(U,t_c) }, 
\end{aligned}
\end{equation}
where $\mathcal{D}(U,t_c)$ is defined in \eqref{eq:denom_err_rep}.
The terms in the numerator and denominator containing the unknown error $e(x,t)$ may be estimated using Theorem \ref{thm:standard_error_rep}. However, {each of these} expressions have different stability properties, {their estimation requires} the solution to distinct adjoint problems as given below.

\bigskip
\noindent \textbf{First adjoint problem} To estimate $\left(w(\cdot),e(\cdot,t_c)\right)$, \eqref{eq:strong_adj} is solved backwards from $t_c$ with initial condition $\psi(x)=\psi^{(1)}(x)=w(x)$. Substituting the solution $\phi^{(1)}$ in \eqref{eq:standard_error_rep} provides a computable estimate for
\begin{equation*}
    \mathcal{E}_1 \approx \left(\psi^{(1)}(\cdot),e(\cdot,t_c)\right) = \left(w(\cdot),e(\cdot,t_c)\right).
\end{equation*}

\bigskip
\noindent \textbf{Second adjoint problem} To estimate $\left(L_2w(\cdot),L_1e(\cdot,t_c)\right)$, we assume $w(x)$ is smooth enough for $L^*w$ to be well-defined and solve \eqref{eq:strong_adj} backwards from $t_c$ with initial condition $\psi(x)=\psi^{(2)}(x)=L^*w(x)$. Substituting the solution $\phi^{(2)}$ in \eqref{eq:standard_error_rep} provides a computable estimate for
\begin{equation*}
\mathcal{E}_2 \approx \left( \psi^{(2)}(\cdot), e(\cdot,t_c) \right) = \left( L^* w(\cdot), e(\cdot,t_c) \right) = \left( L_2 w(\cdot), L_1 e(\cdot,t) \right).
\end{equation*}

\bigskip
\noindent \textbf{Third adjoint problem} To estimate $\left(\nabla_uf^\top(U,\cdot,t_c)w(\cdot),e(\cdot,t_c)\right)$, \eqref{eq:strong_adj} is solved backwards  with initial condition $\psi(x)=\psi^{(3)}(x)=\nabla_uf^\top(U,x,t_c)w(x)$. Substituting the solution $\phi^{(3)}$ in \eqref{eq:standard_error_rep} provides a computable estimate for
\begin{equation*} \label{eq:third_adj}
    \mathcal{E}_3 \approx \left(\psi^{(3)}(\cdot),e(\cdot,t_c)\right) = \left(\nabla_uf^\top(U,\cdot,t_c)w(\cdot),e(\cdot,t_c)\right).
\end{equation*}
The third adjoint problem is only required for nonlinear problems. If the right side in \eqref{eq:semilinear_evo_PDE} is $f=f(x,t)$, the gradient is $\nabla_u f=0$ and  $\mathcal{E}_3 \equiv 0$.

\begin{remark}
If the numerical solution $U$ is not sufficiently accurate, the $n$-th event for the approximate solution $U$ may be closest to the $m$-th event for the true solution $u$ where $n\neq m$. If this is the case, approximation by Taylor series around $t_c$ and $U(x,t_c)$ is inappropriate and the error estimator may not be accurate. See Remark 2 in \cite{ChaEstSteTav21}.
\end{remark}


\section{Model problems}
\label{sec:Model_Problems}

We apply the abstract theory developed in \S \ref{sec:Error_Estimate} to  the heat equation and the 1D linearized shallow water equations.

\subsection{Heat equation}
\label{sec:Heat_Equation}

The heat equation models the diffusion of heat or chemical species. We consider the heat equation in a domain $\Omega \subset \mathbb{R}^n$,
\begin{equation} \label{eq:heat}
\begin{aligned}
u_t(x, t) - \nabla^2 u(x, t) &= f(u,x, t), \qquad &&x \in \Omega, \qquad t\in (0,T],\\
u(x,t) &= 0, \qquad &&x \in \partial\Omega,  \ \quad {t \in  (0,T],}\\
u(x,0) &= u_0(x), \qquad &&x \in \Omega,
\end{aligned}
\end{equation}
where $u(x,t)$ is the temperature of the medium or concentration of a chemical species.

\bigskip
\noindent \underline{Weak form}

The weak form is: Find $u \in L^2(0,T;H_0^1(\Omega))$ such that
\begin{equation*} \label{eq:heat_weak}
\left( u_t, v \right) + \left( \nabla u, \nabla v \right)
  = \left( f, v \right) \quad \forall v \in H_0^1(\Omega), \qquad t \in (0,T].
\end{equation*}
When written in standard form as $u_t + Lu=f$, the operator $L$ is $-\nabla^2$. For any function $v \in H_0^1(\Omega)$,
\begin{equation*}
\left( Lu, v \right) = \left( -\nabla^2 u, v \right) = \left( \nabla u, \nabla v \right),
\end{equation*}
implying
\begin{equation*}
L_1 = \nabla = L_2.
\end{equation*}

\bigskip
\noindent \underline{Error estimate}

The error estimate \eqref{eq:error_estimate_computable} applied to the heat equation \eqref{eq:heat} becomes
\begin{equation*} \label{eq:heat_error}
t_t - t_c \approx
  \frac{\left(w(\cdot), e(\cdot, t_c)\right)}
       {\left( \nabla w(\cdot), \nabla U(\cdot, t_c) \right)  - \left( w(\cdot), f(U,\cdot, t_c) \right) - \left(\nabla_uf^\top(U,\cdot,t_c)w(\cdot),e(\cdot,t_c) \right)
      + \left( \nabla w(\cdot), \nabla e(\cdot, t_c) \right)}.
\end{equation*}
The three adjoint problems presented in \S \ref{sec:Error_Estimate_Computable} all have the same form,
\begin{equation*} \label{eq:heat_adj}
\left.
\begin{aligned}
-\phi^{(i)}_{t}(x, t) - \nabla^2 \phi^{(i)}(x, t) &= \nabla_u f^{\top}  \phi, &&x \in \Omega,  &&&t \in (t_c, 0]\\
\phi^{(i)}(x, t)   &= 0,  && x \in \partial \Omega, &&&t \in (t_c, 0]   \\
\phi^{(i)}(x, t_c) &= \psi^{(i)}(x), &&x \in \Omega
\end{aligned}
\qquad \right\}
\quad i=1,2,3,
\end{equation*}
but different initial conditions:
\begin{equation*} \label{eq:heat_adj_inits}
\begin{aligned}
    \psi^{(1)}(x) &= w(x), \\
    \psi^{(2)}(x) &= {-} \nabla^2 w{(x)}, \\
    \psi^{(3)}(x) &= \nabla_u f^{\top}(U,x,t_c) w(x).
\end{aligned}
\end{equation*}

\subsection{1D linearized shallow water equation (SWE)}
\label{sec:SWE}

The shallow water equations model wave propagation in a fluid domain in which the horizontal scale is large compared to the vertical scale. Applications of the SWE arise frequently in the study of the ocean and atmosphere. In particular, the linearized SWE model have been used to model tsunami wave propagation and inundation in coastal areas, flooding from a dam break, and flows and waves in the atmosphere, see e.g., \cite{Buhler98,CarYeh05}. \textit{A posteriori} error analysis for QoIs that are functionals of the solution to the linearized SWE has been conducted previously in \cite{DavLeV16}.

The nonlinear shallow water equations in one dimension on a domain $\Omega = [0,D]$  are (see \cite{DavLeV16})
\begin{equation} \label{eq:SWE}
\left.
\begin{aligned}
\eta_t(x, t) + \mu_x(x, t) &= f_1(x, t),  \\
\mu_t(x, t) + \left(\frac{\mu(x, t)^2}{\eta(x, t)} + \frac{1}{2}g\eta(x, t)^2\right)_x + g\eta B_x(x) &= f_2(x, t),
\end{aligned}
\qquad
\right \} \qquad x \in \Omega, \qquad t \in (0,T].
\end{equation}
where $\eta(x, t)$ corresponds to the water surface elevation and $\mu(x, t)$ corresponds to the momentum of the fluid. The constant $g$ is the gravitational acceleration constant, $f(x, t) = (f_1(x, t), f_2(x,t))^\top$ is the forcing (typically $0$) and $B(x)$ is the bathymetry (bottom surface profile).

Linearizing \eqref{eq:SWE} around a flat fluid surface with elevation $\bar{\eta}$ and momentum $\bar{\mu} = 0$,
\begin{align} \label{eq:SWE_lin}
\left.
\begin{aligned}
\eta_t(x, t) + \mu_x(x, t) &= f_1(x, t), \\
\mu_t(x, t) + g\bar{h}(x)\eta_x(x, t) &= f_2(x, t),\\
\end{aligned}
\qquad
\right \} \qquad &x \in \Omega, \hspace{2mm}t \in (0,T], \nonumber\\
\left.
\begin{aligned}
\mu(x,t) &= 0, \hspace{5.6mm} 
\end{aligned} 
\qquad
\right \} \hspace{7.75mm} &x \in \partial \Omega, t \in (0,T], \\
\left.
\begin{aligned}
\eta(x,0)&=\eta_0(x), \hspace{2.75mm}\\
\mu(x,0)&=\mu_0(x), \\
\end{aligned}
\qquad
\right \} \qquad &x \in \Omega,\nonumber
\end{align}
where $\bar{h}(x) = \bar{\eta} - B(x)$, $\eta_0(x)$ is the initial surface profile, {$\mu_0(x)$ is the initial momentum profile}, and { $\eta$ satisfies reflective boundary conditions} (see \cite{DavLeV16}).

\bigskip
\noindent \underline{Weak form}

The weak form of \eqref{eq:SWE_lin} is: Find $\eta \in L^2(0,T;H^1(\Omega))$ and $\mu \in L^2(0,T;H_0^1(\Omega))$
\begin{equation*} \label{eq:SWE_weak}
\left( \eta_t, v_1 \right) + \left( \mu_t, v_2 \right) + \left(\mu_x, v_1\right) + g \bar{h} \left( \eta_x, v_2 \right)
= \left( f_1, v_1 \right) + \left( f_2, v_2 \right), \quad
\forall {v_1 \in H^1(\Omega), } \, v_2 \in H^1_0(\Omega), \quad t \in (0,T).
\end{equation*}
When written in standard form as $u_t + Lu=f$, where {$u(x,t)=(\eta(x, t),\mu(x, t))^{\top}$} and the operator $L$ is
\begin{equation*} \label{eq:SWE_L}
L =
\begin{pmatrix}
  0& \frac{\partial}{\partial x}  \\
  g\bar{h}(x)\frac{\partial}{\partial x} & 0
\end{pmatrix},
\end{equation*}
and since $L$ contains only first order derivatives, the operators $L_1$ and $L_2$ are
\begin{equation*}
L_1 = L =
\begin{pmatrix}
0 & \frac{\partial}{\partial x}  \\
g\bar{h}(x)\frac{\partial}{\partial x} & 0
\end{pmatrix},
\qquad
L_2 =
\begin{pmatrix}
1 & 0 \\  0 & 1
\end{pmatrix}.
\end{equation*}

\bigskip
\noindent \underline{Error estimate}

For the linearized shallow water equations \eqref{eq:SWE_lin}, Corollary \ref{cor:lin_err} is applicable and the error estimate \eqref{eq:error_estimate_computable} becomes
\begin{equation*} \label{eq:SWE_error}
t_t - t_c \approx
  \frac{\left(w(\cdot), e(\cdot, t_c)\right)}
       {\left( w(\cdot), A(\cdot) U_x(\cdot, t_c)\right)  - \left( w(\cdot), f(\cdot, t_c) \right)
      + \left( w(\cdot), A(\cdot) e_x(\cdot, t_c)\right)},
\end{equation*}
where
\begin{equation*}
u =    \begin{pmatrix} \eta \\ \mu \end{pmatrix}, \quad
f =    \begin{pmatrix} f_1 \\ f_2 \end{pmatrix}, \quad
A(x) = \begin{pmatrix} 0 & 1 \\ g\bar{h}(x) & 0 \end{pmatrix}.
\end{equation*}
The two non-trivial adjoint problems take the form
\begin{equation*} \label{eq:SWE_adj}
\left.
\begin{aligned}
{-}\phi^{(i)}_{t}(x,t) {-(} A^{\top}(x)\phi^{(i)}{)}_{x}(x,t) &= 0,  &&x \in \Omega, &&&t \in (t_c,0], \\
\phi^{(i)}(x, t)   &= 0, \quad &&x \in \partial \Omega, &&&t \in (t_c, 0],  \\
\phi^{(i)}(x, t_c) &= \psi^{(i)}(x),  &&x\in\Omega
\end{aligned}
\qquad \right\}
\quad i=1,2,
\end{equation*}


\noindent where
\begin{equation*}
    \phi^{(i)} = \begin{pmatrix} \phi^{(i)}_1 \\ \phi^{(i)}_2 \end{pmatrix}
\end{equation*}
with initial conditions
\begin{equation*} \label{eq:SWE_adj_inits}
\begin{aligned}
    \psi^{(1)}(x) &= w(x), \\
    \psi^{(2)}(x) &= {-}(A^\top(x)w(x))_x.
\end{aligned}
\end{equation*}


\section{Numerical results}
\label{sec:Numerical_Results}

All numerical examples were computed using a uniform one-dimensional spatial mesh of length $D$ with $N_x$ finite elements and a uniform temporal mesh of length $T$ with $N_t$ elements. The forward problems were solved using a space-time Galerkin method as described \S \ref{sec:Galerkin_Method}. Unless otherwise noted, the same degree of approximation was used for the basis functions in both space and time. In all examples, the adjoint problems were solved using the same space-time mesh as the forward problem, but with a Galerkin finite element method \emph{two} degrees higher in both space and time. We choose to solve the adjoint problem with this precision in order to demonstrate the accuracy of the estimates. In practice, when estimates are used to drive an adaptive strategy, lower accuracy may be acceptable.


\subsection{1D linear heat equation}
\label{sec:Linear_Heat}

Consider the heat equation in one dimension,
\begin{equation} \label{eq:Linear_Heat}
\begin{aligned}
    u_t-u_{xx} &= \sin(\pi x)(\pi^2\cos(t) - \sin(t)), \quad &x \in [0, 1], \quad &t \in (0, 0.5], \\
    u(0, t)    &= u(1, t) = 0, & &t \in (0, 0.5], \\
    u(x, 0)    &= \sin(\pi x), &x \in [0, 1]. \quad &
\end{aligned}
\end{equation}
The right-hand-side forcing function has been chosen so that \eqref{eq:Linear_Heat} has the analytical solution $u(x,t) = \cos(t)\sin(\pi x)$. Let the functional $G(u; t)$ be
\begin{equation*} \label{eq:G_Linear_Heat}
    G(u; t) = \left(w(\cdot),u(\cdot,t) \right) =\left(\sin(\pi(\cdot)),u(\cdot,t) \right),
\end{equation*}
and the threshold value be $R=0.47$. The true values of the functional $G(u; t)$ are provided in Figure \ref{fig:Reference_QoI_Heat_Manufactured} and the threshold value $R$ appears as a dashed horizontal line. The QoI is the first occurrence of the event $G(u;t)=R$, that is, $t_t =  H(u,0)$.
\begin{figure}[ht]
\centering
\subfloat[\S \ref{sec:Linear_Heat} and \S \ref{sec:NonLinear_Heat}.]{
\includegraphics[width=8cm]{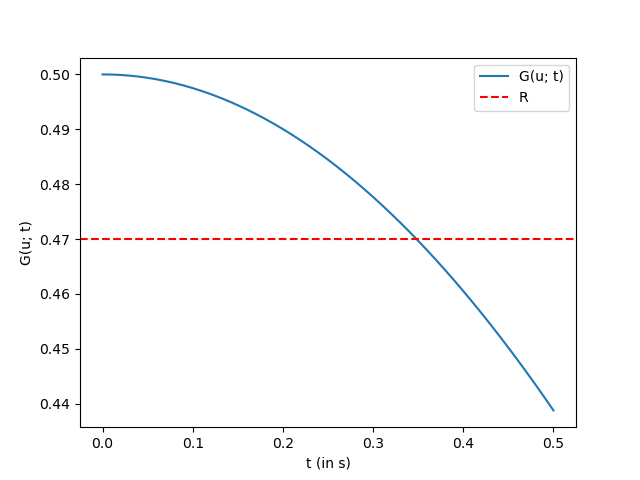}
\label{fig:Reference_QoI_Heat_Manufactured}}
\hfill
\subfloat[\S \ref{sec:SWE_Manufactured}.]{\includegraphics[width=8cm]{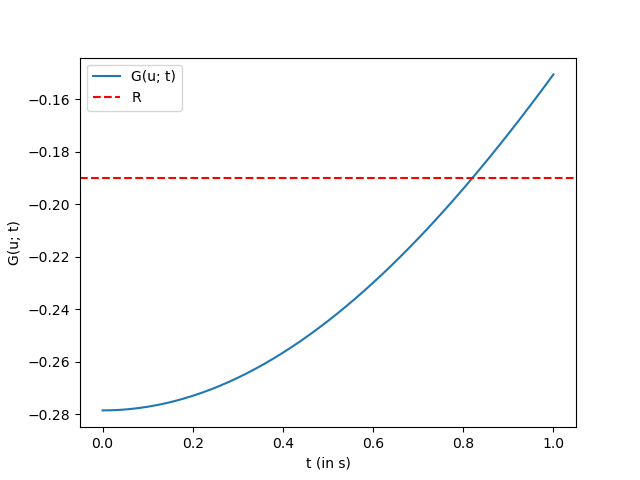}
\label{fig:Reference_QoI_SWE_Manufactured}}
\caption{Reference values of the functional $G(u; t)$ for \S \ref{sec:Linear_Heat}, \S \ref{sec:NonLinear_Heat}, and \S \ref{sec:SWE_Manufactured}.}
\end{figure}
The ``initial'' conditions for the adjoint problems in \S \ref{sec:Heat_Equation} are
\begin{equation*}
\begin{aligned}
\psi^{(1)}(x) &= w(x) = \sin(\pi x), \\
\psi^{(2)}(x) &= {-}w_{xx}{(x)} = \pi^2\sin(\pi x), \\
\psi^{(3)}(x) &= \nabla_u f^\top(U,x,t_c)w(x) = 0.
\end{aligned}
\end{equation*}
The problem \eqref{eq:Linear_Heat} was solved using the cG(1,1) method on a sequence of increasingly refined uniform meshes. The computed time $t_c$ to the first occurrence of the event $G(u;t)=R$, the error in the QoI $e_Q$, and the effectivity ratio $\rho_{\rm eff}$, of the \emph{a posteriori} error estimator are presented in Table \ref{tab:Linear_Heat}. The error decreases as the mesh is refined as expected, and the effectivity ratios are close to 1 in all cases.

\begin{table}[H]
    \centering
    \begin{tabular}{||c||c|c|c||}
    \hline
         $N$  & $t_c$ & $e_Q$ & $\rho_{\rm eff}$ \\
         \hline
         50  & 0.346346 & 1.820 $\times 10^{-3}$ & 1.003 \\
         100 & 0.347711 & 4.546 $\times 10^{-4}$ & 1.001 \\
         200 & 0.348053 & 1.129 $\times 10^{-4}$ & 1.000 \\
         400 & 0.348138 & 2.839 $\times 10^{-5}$ & 1.000 \\
         \hline
    \end{tabular}
    \caption{Computed time $t_c$ to the first occurrence of the event $G(u;t)=R$,  error in the QoI $e_Q$, and  effectivity ratio $\rho_{\rm eff}$, for \S \ref{sec:Linear_Heat}, using cG(1,1) to solve the forward problem and cG(3,3) to solve the adjoint problems. The exact time to the first occurrence of the event is $t_t = 0.34816603$.}
    \label{tab:Linear_Heat}
\end{table}


\subsection{1D nonlinear heat equation}
\label{sec:NonLinear_Heat}

Consider the nonlinear heat equation in one dimension

\begin{equation} \label{eq:NonLinear_Heat}
\begin{aligned}
    u_t - u_{xx} &= -u^2 + \sin(\pi x)(-\sin(t) + \pi^2\cos(t) + \cos^2(t)\sin(\pi x)), \quad &x \in (0, 1], \quad &t \in (0, 0.5], \\
    u(0, t)      &= u(1, t) = 0 & &t \in (0, 0.5], \\
    u(x, 0)      &= \sin(\pi x) &x \in [0, 1]. \quad &
\end{aligned}
\end{equation}
The right-hand-side ``forcing'' function has been chosen so that the nonlinear heat equation \eqref{eq:NonLinear_Heat} has the same analytical solution as the linear heat equation \eqref{eq:Linear_Heat} in \S \ref{sec:Linear_Heat}, namely $u(x,t) = \cos(t)\sin(\pi x)$. The functional $G(u; t)$ is again chosen to be
\begin{equation*} \label{eq:G_NonLinear_Hear}
    G(u; t) = \left(w(\cdot),u(\cdot,t) \right) =\left(\sin(\pi(\cdot)),u(\cdot,t) \right),
\end{equation*}
and the threshold value to be $R=0.47$. The QoI is the first occurrence of the event $G(u;t)=R$, that is, $t_t =  H(u,0)$. The true values of the functional are graphed in Figure \ref{fig:Reference_QoI_Heat_Manufactured}.

The ``initial'' conditions for the adjoint problems in \S \ref{sec:Heat_Equation} are
\begin{equation*}
\begin{aligned}
\psi^{(1)}(x) &= w(x) = \sin(\pi x), \\
\psi^{(2)}(x) &= {-}w_{xx}{(x)} = \pi^2\sin(\pi x), \\
\psi^{(3)}(x) &= \nabla_uf^\top(U,x,t_c)w(x) = -2\sin(\pi x)U(x,t_c).
\end{aligned}
\end{equation*}
Note in particular, that unlike for the previous, linear case in \S \ref{sec:Linear_Heat}, the third adjoint problem is non-trivial.
The problem \eqref{eq:NonLinear_Heat} was solved using the cG(1,1) methods using a sequence of increasingly refined uniform meshes. The computed time $t_c$ to the first occurrence of the event $G(u;t)=R$, the error in the QoI $e_Q$, and the effectivity ratio $\rho_{\rm eff}$,  of the \emph{a posteriori} error estimator are presented in Table \ref{tab:NonLinear_Heat} which clearly indicates that the error decreases as the mesh is refined and the effectivity ratios remains close to 1 in all cases.

\begin{table}[H]
    \centering
    \begin{tabular}{||c||c|c|c||}
    \hline
         $N$ & $t_c$ & $e_Q$ & $\rho_{\rm eff}$ \\
         \hline
         50  & 0.346531 & 1.635 $\times 10^{-3}$ & 1.002 \\
         100 & 0.347757 & 4.087 $\times 10^{-4}$ & 1.001 \\
         200 & 0.348065 & 1.015 $\times 10^{-4}$ & 1.000 \\
         400 & 0.348140 & 2.553 $\times 10^{-5}$ & 1.000 \\
         \hline
    \end{tabular}
    \caption{Computed time $t_c$ to the first occurrence of the event $G(u;t)=R$, error in the QoI $e_Q$, and  effectivity ratio $\rho_{\rm eff}$, for \S \ref{sec:NonLinear_Heat}, using cG(1,1) to solve the forward problem and cG(3,3) to solve the adjoint problems. The exact time to the first occurrence of the event is $t_t = 0.34816603$.}
    \label{tab:NonLinear_Heat}
\end{table}


\subsection{1D linearized SWE: Manufactured solution}
\label{sec:SWE_Manufactured}

Consider the linearized SWEs \eqref{eq:SWE_lin} over the space-time domain $ (x,t)\in [0, 10]\times[0, 1]$, with right hand side
\begin{equation}
\label{eq:SWE_Manufactured_forcing}
f = \begin{pmatrix} f_1 \\ f_2 \end{pmatrix} =
    \begin{pmatrix}
        -\sin(t)\sin(\pi x) + \pi\cos(t)\cos(\pi x) \\
        -\sin(t)\sin(\pi x) + \pi g \bar{h}(x)\cos(t)\cos(\pi x)
    \end{pmatrix}.
\end{equation}
In order to construct an analytical exact solution, the bathymetry was chosen to be a constant, $B(x)= -10$ as shown in Figure \ref{fig:SWE_Manufactured_Bathymetry} and Dirichlet boundary conditions were imposed at the boundaries of the domain, specifically $u(0, t) = u(10, t) = (2,0)^{\top}$. The initial surface elevation and momentum were chosen as $u(x, 0) = (\eta_0(x), \mu_0(x))^{\top} = ( 2 + \sin(\pi x) , \sin(\pi x) )^{\top} $. The initial surface elevation $\eta_0(x)$ is shown in Figure \ref{fig:SWE_Manufactured_InitialWaveHeight}. The rest-state elevation was $\bar{\eta} = 2$.

\begin{figure}[ht]
\centering
\subfloat[$B(x)$]{\includegraphics[width=5cm]{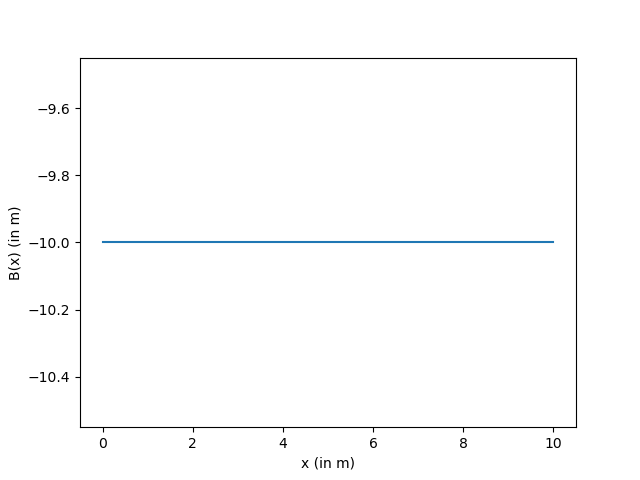}
\label{fig:SWE_Manufactured_Bathymetry}}
\subfloat[$\eta_0(x)$]{\includegraphics[width=5cm]{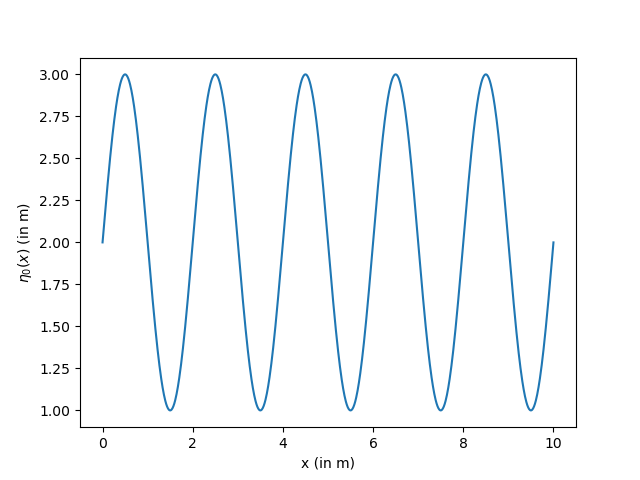}
\label{fig:SWE_Manufactured_InitialWaveHeight}}
\subfloat[$w_2(x)$]{\includegraphics[width=5cm]{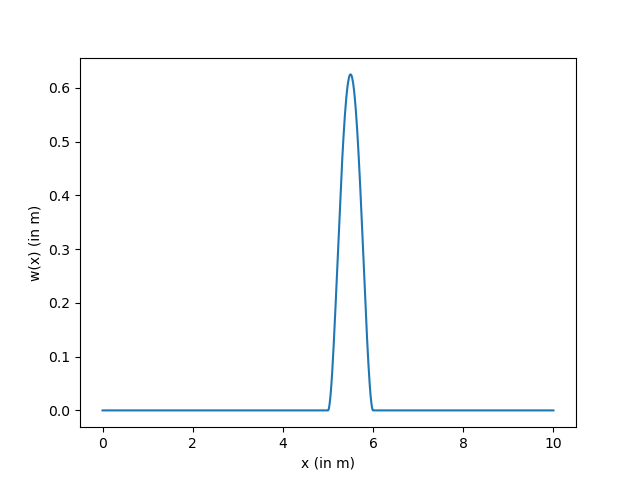}
\label{fig:SWE_Manufactured_WeightFunction}}
\caption{The bathymetry $B(x)$, initial wave height $\eta_0(x)$, and nonzero component of the weight function $w_2(x)$ for \S \ref{sec:SWE_Manufactured}.}
\end{figure}

Finally, the forcing term \eqref{eq:SWE_Manufactured_forcing} was chosen so that \eqref{eq:SWE_lin} has the analytical solution
\begin{equation*}
u(x, t) =
    \begin{pmatrix} \eta(x, t) \\ \mu(x, t) \end{pmatrix} =
    \begin{pmatrix} 2 + \cos(t)\sin(\pi x) \\ \cos(t)\sin(\pi x) \end{pmatrix}.
\end{equation*}
In order to define the functional $G(u;t)$, let the weight function be $w(x) = (0,w_2(x))$  where
\begin{equation*}
w_2(x) =
\qquad \left \{ \qquad
\begin{aligned}
0  \qquad\qquad\qquad          & x \leq 5, \\
10(x - 5)^2(x - 6)^2 \quad 5 < & x < 6, \\
0  \qquad\qquad\qquad          & x \geq 6.
\end{aligned}
\qquad \right \}
\end{equation*}
The non-zero component of the weight function $w_2(x)$ is depicted in Figure \ref{fig:SWE_Manufactured_WeightFunction}.
The true values of the functional $G(u; t)= \left(w(\cdot),u(\cdot,t)\right)$ are provided in Figure \ref{fig:Reference_QoI_SWE_Manufactured}, where the threshold value $R=-0.19$ appears as a dashed horizontal line. The QoI is the first occurrence of the event $G(u;t)=R$, that is, $t_t =  H(u,0)$.

The initial conditions  for the adjoint problems  in \S \ref{sec:SWE} are
\begin{equation*}
\begin{aligned}
    \psi^{(1)}(x) &= (0,w_2(x))^{\top}, \\
    \psi^{(2)}(x) &= ({-}g(\bar{h}(x)w_2(x))_x,0)^{\top}.
\end{aligned}
\end{equation*}
The linear SWE were solved using the cG(2,2) method (\S \ref{sec:Galerkin_Method}) on a sequence of increasingly refined uniform meshes. The computed time $t_c$ to the first occurrence of the event $G(u;t)=R$, the error in the QoI $e_Q$, and the effectivity ratio $\rho_{\rm eff}$, of the \emph{a posteriori} error estimator are presented in Table \ref{tab:SWE_Manufactured}.
\begin{table}[H]
    \centering
    \begin{tabular}{||c||c|c|c|c||}
    \hline
         $N$ & $t_c$ & $e_Q$ & $\rho_{\rm eff}$ \\
         \hline
         50  & 0.820996 & -1.120 $\times 10^{-3}$ & 0.997 \\
         100 & 0.819917 & -4.035 $\times 10^{-5}$ & 1.000 \\
         200 & 0.819883 & -6.940 $\times 10^{-6}$ & 1.000 \\
         400 & 0.819878 & -1.851 $\times 10^{-6}$ & 1.000 \\
         \hline
    \end{tabular}
    \caption{Computed time $t_c$ to the first occurrence of the event $G(u;t)=R$, error in the QoI $e_Q$, and  effectivity ratio $\rho_{\rm eff}$, for \S \ref{sec:SWE_Manufactured}, using cG(2,2) to solve the forward problem and cG(4,4) to solve the adjoint problems. The exact time to the first occurrence of the event is $t_t = 0.81987644$. }
    \label{tab:SWE_Manufactured}
\end{table}


\subsection{1D linearized SWE: Constant bathymetry}
\label{sec:SWE_Constant}

Consider the linearized SWEs \eqref{eq:SWE_lin} over the space-time domain $(x,t) \in [0, 400] \times [0, 200]$, with the more realistic right hand side $f =(0,0)^{\top} $. The bathymetry was chosen to be the constant $B(x)= -0.1$ as depicted in Figure \ref{fig:SWE_Constant_Bathymetry} and the rest-state height was $\bar{\eta} = 1$. Dirichlet boundary conditions were imposed on the momentum only, specifically $ \mu(0, t) = \mu(D, t) = 0$. The initial conditions were $u(x,0)=(\eta_0(x),0)^{\top}$, where
\begin{equation*}\label{eq:initcondswe}
\eta_0(x) =
\qquad \left \{ \qquad
\begin{aligned}
        0,  \qquad\qquad\qquad\qquad                            & x \leq 100, \\
        \frac{0.4}{390625}(x - 100)^2(x - 150)^2,  \quad  100 < & x < 150, \\
        0,  \qquad\qquad\qquad\qquad                            & x \geq 150.
    \end{aligned}
\qquad \right \}
\end{equation*}
The initial surface elevation $\eta_0(x)$ is shown in Figure \ref{fig:SWE_Constant_InitialWaveHeight}.

In order to define the functional $G(u;t)$ we set the weight function $w=(w_1(x),0)^{\top}$ where
\begin{equation*}
w_1(x) =
\qquad \left \{ \qquad
\begin{aligned}
        0,   \qquad\qquad\qquad\qquad                        & x \leq 160, \\
        \frac{1}{200000}(x - 160)^2(x - 200)^2, \quad  160 < & x < 200, \\
        0,   \qquad\qquad\qquad\qquad                        & x \geq 200.
    \end{aligned}
\qquad \right \}
\end{equation*}
The non-zero component of the weight function $w_1(x)$ is shown in Figure \ref{fig:SWE_Constant_WeightFunction}.
\begin{figure}[ht]
\centering
\subfloat[$B(x)$]{
\includegraphics[width=5.25cm]{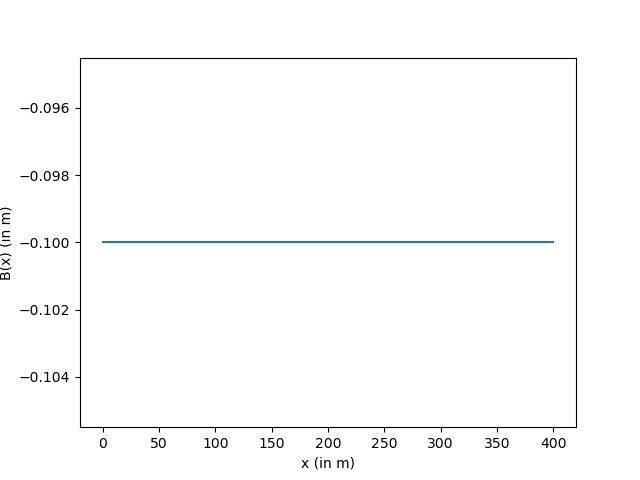}
\label{fig:SWE_Constant_Bathymetry}}
\hfill
\subfloat[$\eta_0(x)$]{
\includegraphics[width=5.25cm]{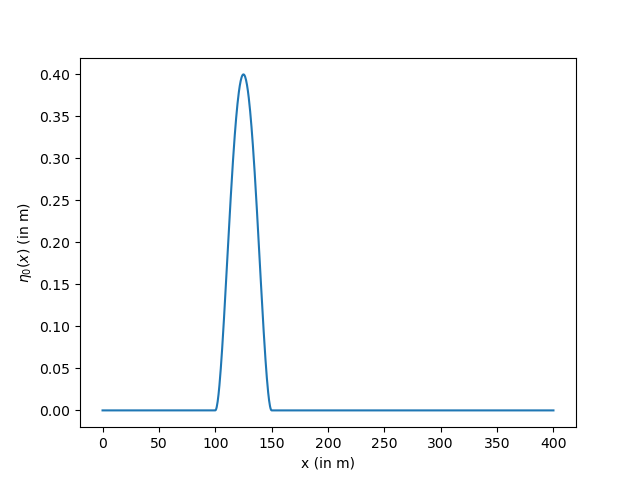}
\label{fig:SWE_Constant_InitialWaveHeight}}
\hfill
\subfloat[$w_1(x)$]{
\includegraphics[width=5.25cm]{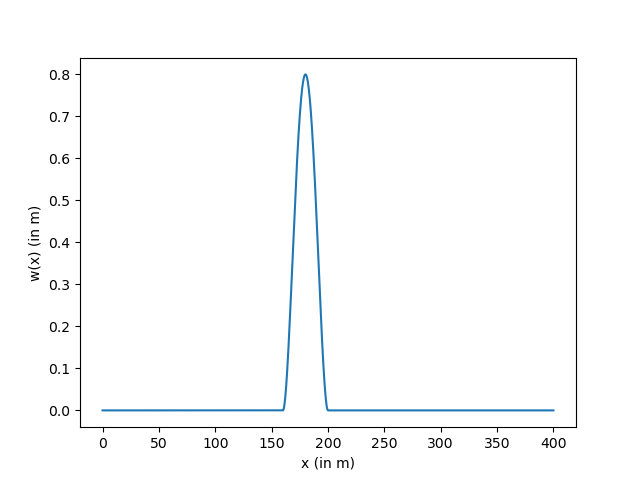}
\label{fig:SWE_Constant_WeightFunction}}
\caption{The bathymetry $B(x)$, initial wave height $\eta_0(x)$, and nonzero component of the weight function $w_1(x)$ for \S \ref{sec:SWE_Constant}.}
\end{figure}
The linear SWE were solved using the cG(2,2) method (\S \ref{sec:Galerkin_Method}), and the surface elevation $\eta(x,t)$ at several time is shown in Figure \ref{fig:wave_height_behavior}.
\begin{figure}[ht]
\centering
\subfloat[$t = 0$]{\includegraphics[width=5.25cm]{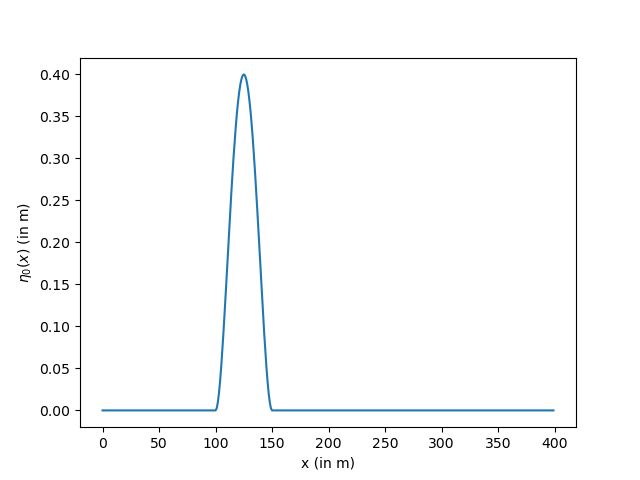} \label{fig:WaveHeight1} }
\subfloat[$t = 20$s]{\includegraphics[width=5.25cm]{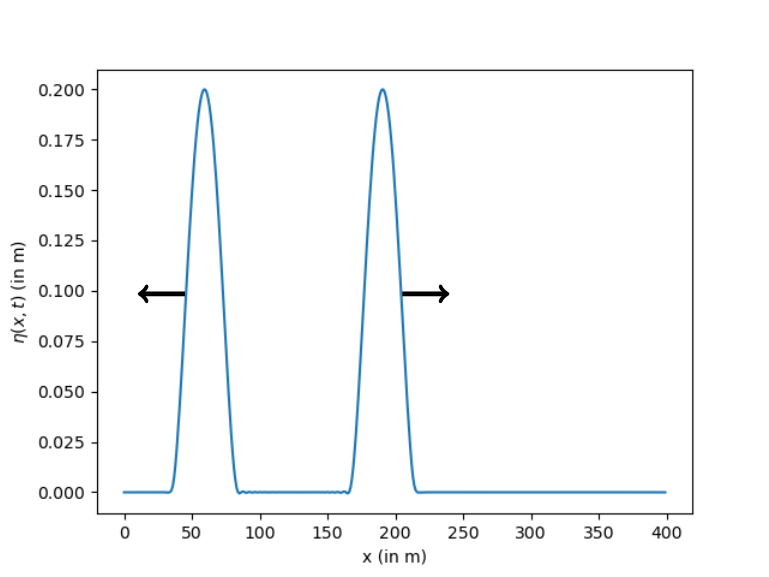} \label{fig:WaveHeight2}}
\subfloat[$t = 40$s ]{\includegraphics[width=5.25cm]{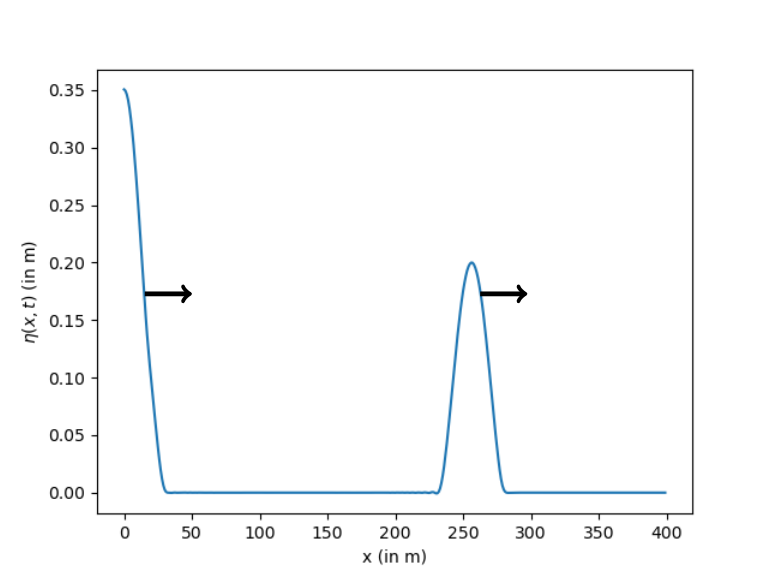} \label{fig:WaveHeight3}}
\newline
\subfloat[$t = 65$s]{\includegraphics[width=5.25cm]{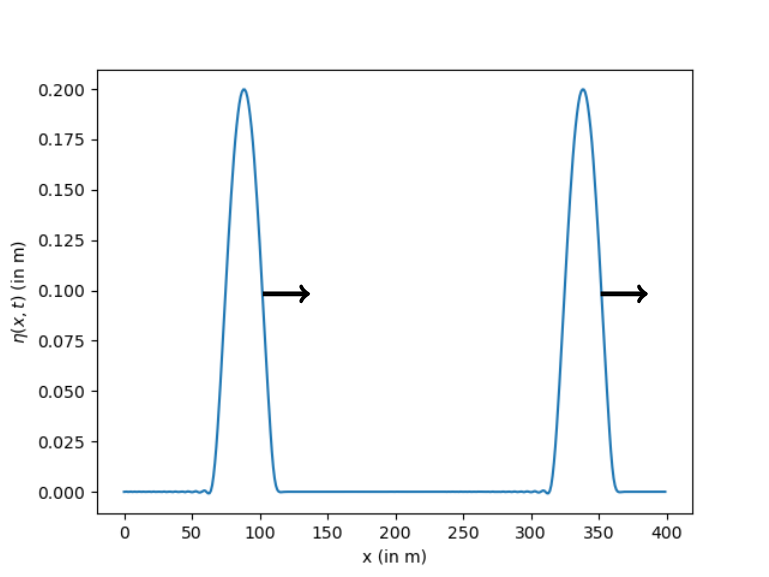} \label{fig:WaveHeight4}}
\subfloat[$t = 90$s]{\includegraphics[width=5.25cm]{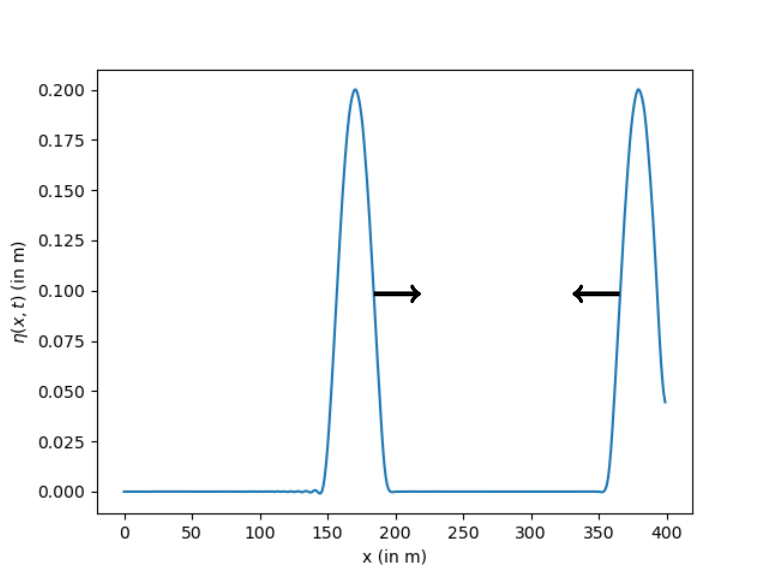} \label{fig:WaveHeight5}}
\subfloat[$t = 107.5$s]{\includegraphics[width=5.25cm]{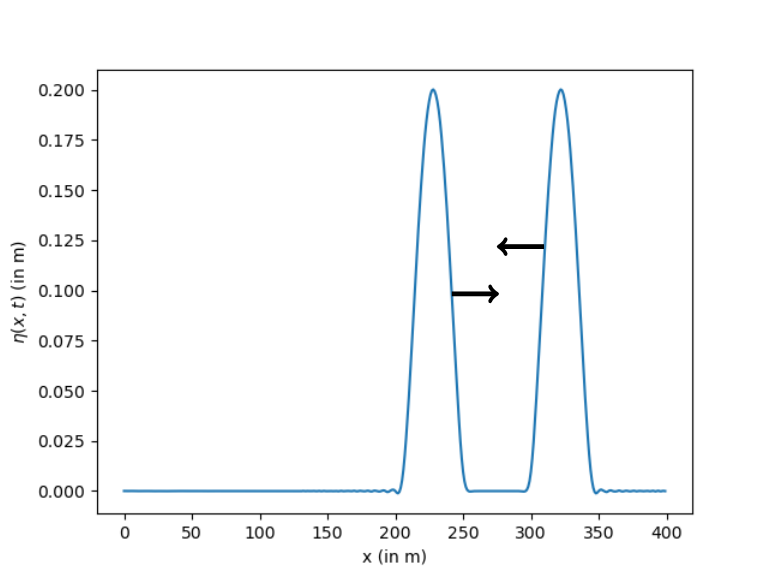} \label{fig:WaveHeight6}}
\newline
\subfloat[$t = 125$s]{\includegraphics[width=5.25cm]{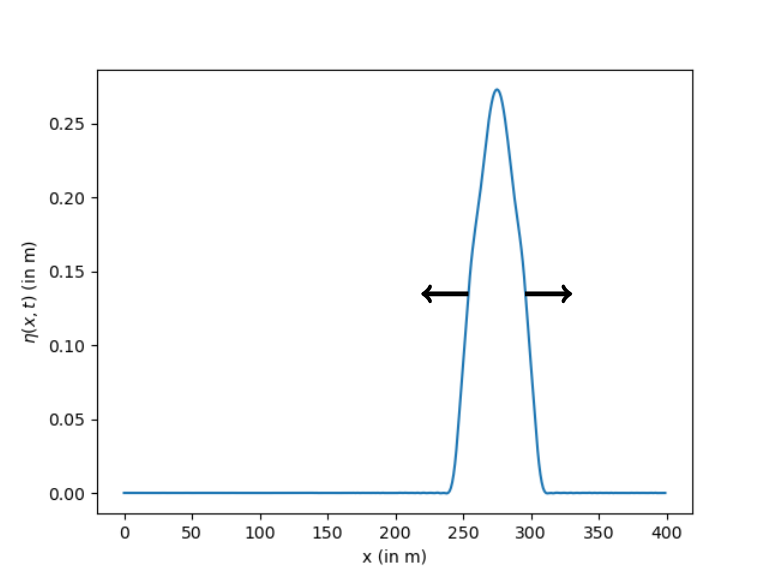} \label{fig:WaveHeight7}}
\subfloat[$t = 165$s]{\includegraphics[width=5.25cm]{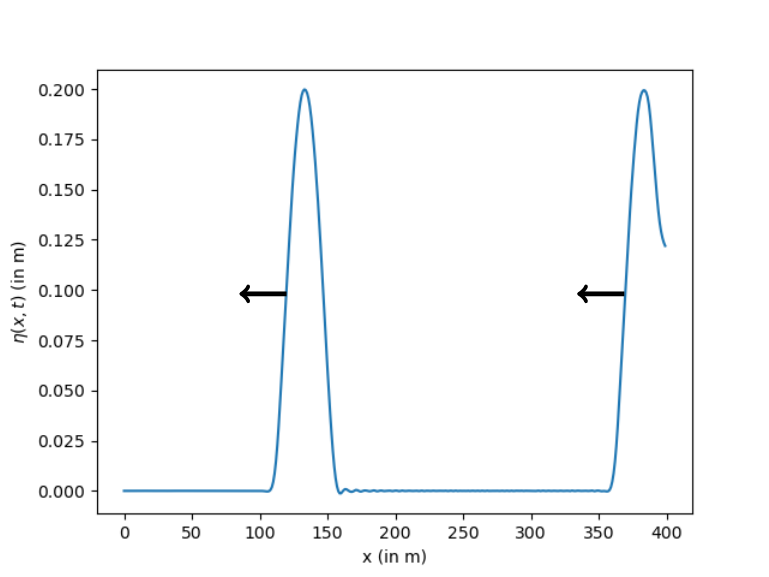} \label{fig:WaveHeight8}}
\subfloat[$t = 200$s]{\includegraphics[width=5.25cm]{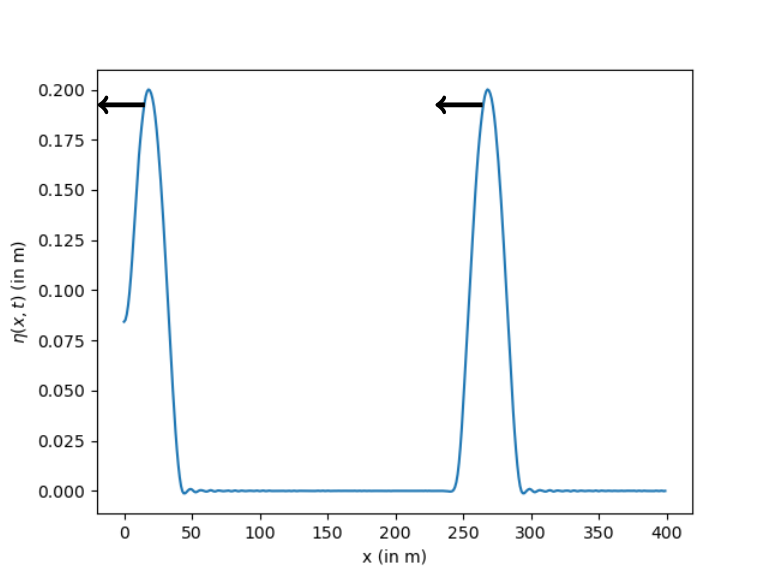} \label{fig:WaveHeight9}}
\caption{Evolution of the surface elevation $\eta(x,t)$ for \S \ref{sec:SWE_Constant}. The initial profile generates left and right traveling waves which eventually reflect off the left-hand and right-hand boundaries. The arrows indicate the direction of propagation.}
\label{fig:wave_height_behavior}
\end{figure}
Since an analytical solution $u$ is unavailable, a reference solution $u$ was computed using a mesh-size of $N_{\rm ref}=800$ with the cG(3,3) method. Reference values of the functional $G(u; t)$ are graphed in Figure \ref{fig:Reference_QoI_SWE_Constant}, where the threshold value of $R=2$ is indicated by the dashed horizontal line

\begin{figure}[ht]
\centering
\subfloat[\S \ref{sec:SWE_Constant}.]{
\includegraphics[width=5.25cm]{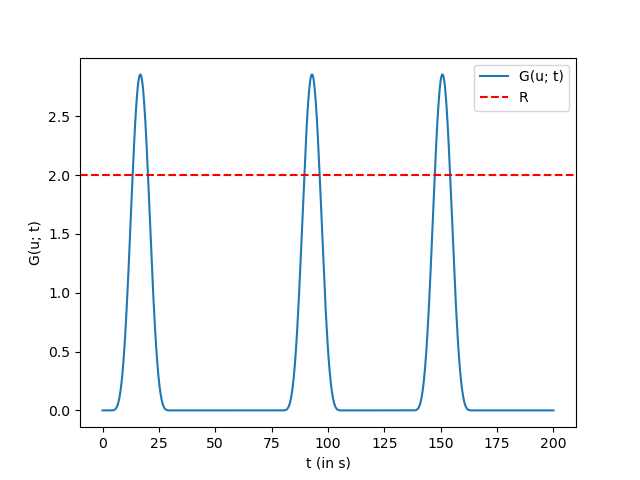}
\label{fig:Reference_QoI_SWE_Constant}}
\hfill
\subfloat[\S \ref{sec:SWE_Shelf}]{
\includegraphics[width=5.25cm]{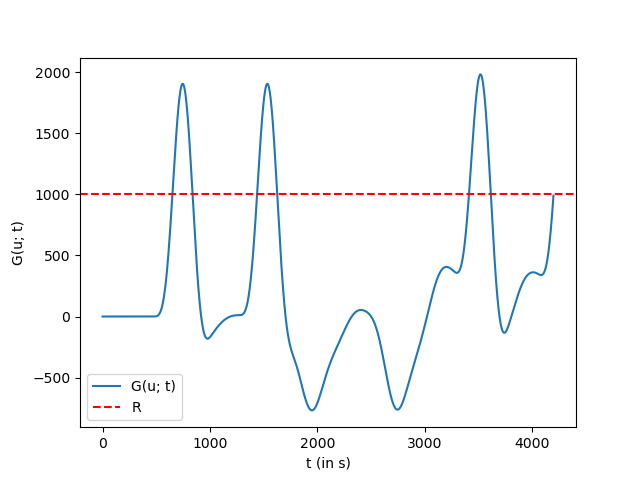}
\label{fig:Reference_QoI_SWE_Shelf}}
\hfill
\subfloat[\S \ref{sec:SWE_Reef}.]{
\includegraphics[width=5.25cm]{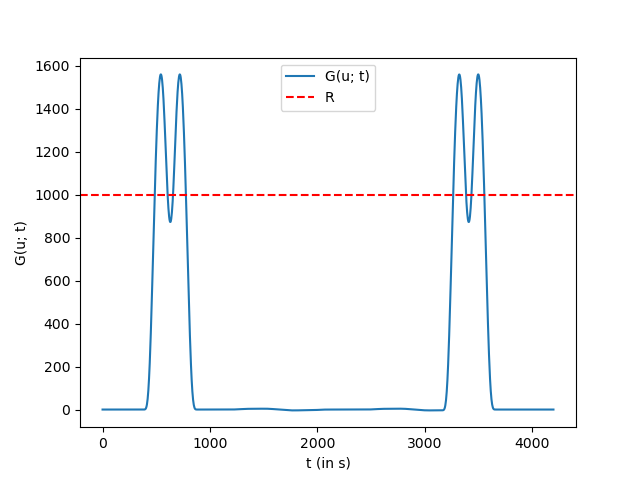}
\label{fig:Reference_QoI_SWE_Reef}}
\caption{Reference values of the $G(u; t)$ for \S \ref{sec:SWE_Constant}, \S \ref{sec:SWE_Shelf} and \S \ref{sec:SWE_Reef}.}
\end{figure}


The QoI defined by the functional $G(u;t)$ is essentially the time at which the surface elevation at a narrow region centered at $x=180$ achieves a specified value. Recall that by an event we mean $G(\cdot;t)$ achieving the threshold value $R$. The first event occurs when the incident wave arrives near $x=180$ and $G(u;t)$ increases above the threshold value. The second event occurs when the incident wave travels beyond the narrow region centered at $x=180$ and $G(u;t)$ drops below the threshold value. The third and fourth events occur when the reflected wave passes through the narrow region centered at $x=180$ and $G(u;t)$ first increases above and then drops below the threshold value.

The initial conditions  for the adjoint problems  in \S \ref{sec:SWE} are
\begin{equation*}
\begin{aligned}
    \psi^{(1)}(x) &= (w_1(x),0)^{\top},     \\
    \psi^{(2)}(x) &= (0,{-}w_{1,x}(x))^{\top}.
\end{aligned}
\end{equation*}
The computed times $t_c$ to the first, second and third occurrences of the event $G(u;t)=R$, errors in the QoI $e_Q$, and effectivity ratios $\rho_{\rm eff}$, are presented in Table \ref{tab:SWE_Constant}, where the effectivity ratios are all seen to be close to 1 even for the later events.
\begin{table}[H]
    \centering
    \begin{tabular}{||c||c|c|c||c|c|c||c|c|c||}
    \hline
         & \multicolumn{3}{c||}{First Event} & \multicolumn{3}{c||}{Second Event} & \multicolumn{3}{c||}{Third Event} \\
         \hline
         \hline
         $N$ & $t_c$ & $e_Q$ & $\rho_{\rm eff}$ & $t_c$ & $e_Q$ & $\rho_{\rm eff}$ & $t_c$ & $e_Q$ & $\rho_{\rm eff}$ \\
         \hline
         50  & 13.4083 & -5.849 $\times 10^{-2}$ & 1.001 & 20.2736 & -1.204 $\times 10^{-1}$ & 1.008 & 90.0299 & -5.370 $\times 10^{-1}$ & 0.991 \\
         100 & 13.3634 & -1.363 $\times 10^{-2}$ & 1.000 & 20.1599 & -6.778 $\times 10^{-3}$ & 1.003 & 89.5458 & -5.291 $\times 10^{-2}$ & 1.000 \\
         200 & 13.3495 &  2.779 $\times 10^{-4}$ & 1.000 & 20.1529 &  2.538 $\times 10^{-4}$ & 0.990 & 89.4961 & -3.164 $\times 10^{-3}$ & 0.988 \\
         400 & 13.3499 & -1.459 $\times 10^{-4}$ & 1.000 & 20.1531 &  6.998 $\times 10^{-5}$ & 1.002 & 89.4931 & -2.166 $\times 10^{-4}$ & 0.967 \\
         \hline
    \end{tabular}
    \caption{Computed times $t_c$ to the first, second and third occurrences of the event $G(u;t)=R$, error in the QoI $e_Q$, and effectivity ratio $\rho_{\rm eff}$, for \S \ref{sec:SWE_Constant}, using cG(2,2) to solve the forward problem and cG(4,4) to solve the adjoint problems. The first, second, and third events occur at $t_t = 13.349799$, $t_t = 20.153171$, and $t_t = 89.492912$ respectively.}
    \label{tab:SWE_Constant}
\end{table}


\subsection{1D linearized SWE: Continental shelf}
\label{sec:SWE_Shelf}

Consider the linearized SWEs \eqref{eq:SWE_lin} over the space-time domain $(x,t) \in [0, 400000] \times [0, 4200]$. In this example, the bathymetry $B(x)$ was chosen to be
\begin{equation*} \label{eq:B_shelf}
B(x) =
\qquad \left \{ \qquad
\begin{aligned}
 -200,    \qquad\qquad\qquad  & x \leq 25000, \\
 -0.152x+3600, \quad  25000 < & x < 50000, \\
 -4000,   \qquad\qquad\qquad  & x \geq 50000,
\end{aligned}
\qquad \right \}
\end{equation*}
as depicted in Figure \ref{fig:SWE_Shelf_Bathymetry} and represents a continental shelf. The rest-state height $\bar{\eta} = 1$. Dirichlet boundary conditions were applied to the momentum only, specifically $ \mu(0, t) = \mu(D, t) = 0$. The initial condition was $u(x,0)=(\eta_0(x),0)^{\top}$, where
\begin{equation*} \label{eq:swelevequeinitcondb}
\eta_0(x) =
\qquad \left \{ \qquad
\begin{aligned}
         0,      \qquad\qquad\qquad\qquad\qquad                            & x \leq 100000, \\       \frac{0.4}{25000^4}(x - 100000)^2(x - 150000)^2,  \quad  100000 < & x < 150000, \\
         0,      \qquad\qquad\qquad\qquad\qquad                            & x \geq 150000.
\end{aligned}
\qquad \right \}
\end{equation*}
The initial surface elevation $\eta_0(x)$ is depicted in Figure \ref{fig:SWE_Shelf_InitialWaveHeight}.

In order to define the functional $G(u;t)$ \eqref{eq:G}, the weight function $w=(w_1(x),0)^{\top}$ was
\begin{equation*}
w_1(x) =
\qquad \left \{ \qquad
\begin{aligned}
        0,    \qquad\qquad\qquad\qquad\qquad                       & x \leq 10000, \\
        \frac{1}{7500^4}(x - 10000)^2(x - 25000)^2, \quad  10000 < & x < 25000, \\
        0,    \qquad\qquad\qquad\qquad\qquad                       & x \geq 25000.
\end{aligned}
\qquad \right \}
\end{equation*}
The non-zero component of the weight function $w_1(x)$ is depicted in Figure \ref{fig:SWE_Shelf_WeightFunction}.

\begin{figure}[ht]
\centering
\subfloat[$B(x)$]{
\includegraphics[width=5.25cm]{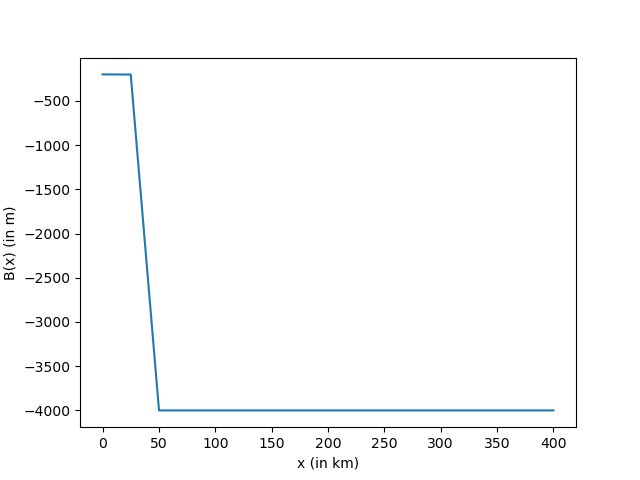}
\label{fig:SWE_Shelf_Bathymetry}}
\hfill
\subfloat[$\eta_0(x)$]{
\includegraphics[width=5.25cm]{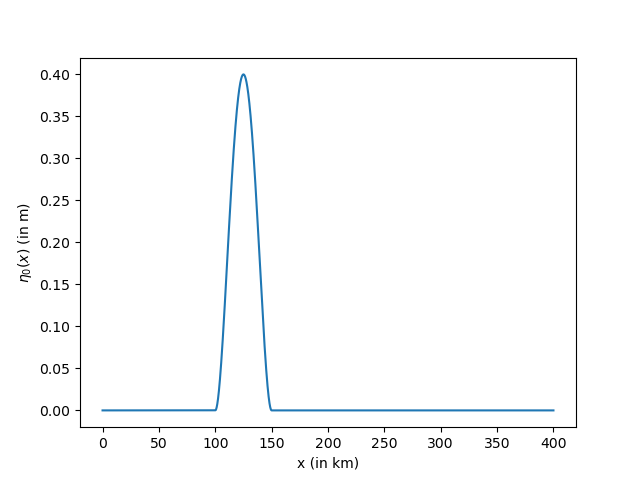}
\label{fig:SWE_Shelf_InitialWaveHeight}}
\hfill
\subfloat[$w_1(x)$]{
\includegraphics[width=5.25cm]{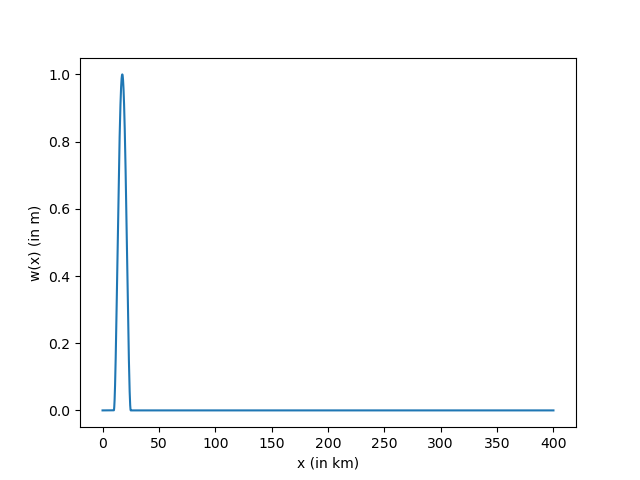}
\label{fig:SWE_Shelf_WeightFunction}}
\caption{The bathymetry $B(x)$, initial wave height $\eta_0(x)$, and nonzero component of the weight function $w_1(x)$ for \S \ref{sec:SWE_Shelf}.}
\end{figure}

A reference solution $u$ was computed using a mesh with $N_{\rm ref}=1280$ sub-intervals with the cG(3,3) method. The corresponding values of the functional $G(u; t)$ are plotted in Figure \ref{fig:Reference_QoI_SWE_Shelf} and the critical value of $R=1000$ is presented as a dashed horizontal line. The QoIs we consider are the times for the first, second and third occurrences of the event $G(u;t)=R$.

The initial conditions  for the adjoint problems  in \S \ref{sec:SWE} are
\begin{equation*}
\begin{aligned}
    \psi^{(1)}(x) &= (w_1(x),0)^{\top}, \\
    \psi^{(2)}(x) &= (0,{-}w_{1,x}(x))^{\top}.
\end{aligned}
\end{equation*}
The forward problem was solved using the cG(2,2) method (\S \ref{sec:Galerkin_Method}) on a sequence of increasingly refined uniform meshes. The computed times $t_c$ to the first, second and third occurrences of the event $G(u;t)=R$, errors in the QoI $e_Q$, and effectivity ratios $\rho_{\rm eff}$, are presented in Table  \ref{tab:SWE_Self_22}.
\begin{table}[H]
    \centering
\begin{tabular}{||c||c|c|c||c|c|c||c|c|c||}
    \hline
         & \multicolumn{3}{c||}{First Event} & \multicolumn{3}{c||}{Second Event} & \multicolumn{3}{c||}{Third Event} \\
         \hline
         \hline
         $N$ & $t_c$ & $e_Q$ & $\rho_{\rm eff}$ & $t_c$ & $e_Q$ & $\rho_{\rm eff}$ & $t_c$ & $e_Q$ & $\rho_{\rm eff}$ \\
         \hline
         80  & 651.267 & -3.647 $\times10^{-1}$ & 1.002 & 838.675 & -2.456 $\times10^{-1}$ & 0.970 & 1438.541 & 9.255 $\times10^{-1}$ & 0.925 \\
         160 & 650.975 & -7.329 $\times10^{-2}$ & 1.000 & 838.421 & 8.597 $\times10^{-3}$ & 0.929 & 1439.384 & 1.008 $\times10^{-1}$ & 0.979 \\
         320 & 650.907 & -4.970 $\times10^{-3}$ & 1.000 & 838.428 & 8.944 $\times10^{-4}$ & 1.006 & 1439.479 & 6.161 $\times10^{-3}$ & 0.986 \\
         640 & 650.902 & 3.672  $\times10^{-4}$ & 1.000 & 838.429 & -2.890 $\times10^{-4}$ & 1.012 & 1439.484 & 1.046 $\times10^{-3}$ & 1.042\\
         \hline
    \end{tabular}
    \caption{Computed times $t_c$ to the first, second and third occurrences of the event $G(u;t)=R$, error in the QoI $e_Q$, and effectivity ratio $\rho_{\rm eff}$, for \S \ref{sec:SWE_Shelf}, using cG(2,2) to solve the forward problem and cG(4,4) to solve the adjoint problems. The first, second, and third events occur at  $t_t = 650.79647$, $t_t = 838.42910$, and $t_t = 1439.4849$ respectively.}
    \label{tab:SWE_Self_22}
\end{table}

The effect of the spatial discretization was investigated by repeating the above calculations using cG(2,1). Note that the degree of the spatial basis functions used to solve the adjoint problems was also reduced by one. As anticipated, as is shown in Table \ref{tab:SWE_Self_21}, the errors $e_Q$ for the times to the first three events $G(u;t)=R$ are much larger than when solving using cG(2,2). However, the effectivity ratios remain close to 1 as desired.
\begin{table}[H]
    \centering
    \begin{tabular}{||c||c|c|c||c|c|c||c|c|c||}
    \hline
         & \multicolumn{3}{c||}{First Event} & \multicolumn{3}{c||}{Second Event} & \multicolumn{3}{c||}{Third event} \\
         \hline
         \hline
         $N$ & $t_c$ & $e_Q$ & $\rho_{\rm eff}$ & $t_c$ & $e_Q$ & $\rho_{\rm eff}$ & $t_c$ & $e_Q$ & $\rho_{\rm eff}$ \\
         \hline
         80  & 643.874 & 6.922  $\times 10^{0}$  & 1.023 & 846.194 & -7.681 $\times 10^{0}$  & 0.999 & 1465.863 & -2.648 $\times 10^{1}$  & 1.050 \\
         160 & 650.852 & -5.536 $\times 10^{-2}$ & 1.014 & 840.674 & -2.161 $\times 10^{0}$  & 0.999 & 1442.571 & -3.192 $\times 10^{0}$  & 0.994 \\
         320 & 650.802 & -5.360 $\times 10^{-3}$ & 1.002 & 838.683 & -1.702 $\times 10^{-1}$ & 0.998 & 1439.675 & -2.964 $\times 10^{-1}$ & 1.003 \\
         640 & 650.795 & 1.889  $\times 10^{-3}$ & 0.997 & 838.521 & -8.420 $\times 10^{-3}$ & 0.995 & 1439.392 & -1.001 $\times 10^{-2}$ & 1.002 \\
         \hline
    \end{tabular}
    \caption{Computed times $t_c$ to the first, second and third occurrences of the event $G(u;t)=R$, error in the QoI $e_Q$, and effectivity ratio $\rho_{\rm eff}$, for \S \ref{sec:SWE_Shelf}, using cG(2,1) to solve the forward problem and cG(4,3) to solve the adjoint problems. The first, second, and third events occur at $t_t = 650.79647$, $t_t = 838.51300$, and $t_t = 1439.3789$ respectively.}
    \label{tab:SWE_Self_21}
\end{table}


\subsection{1D linearized SWE: Reef}
\label{sec:SWE_Reef}

Consider the homogeneous, linearized SWEs \eqref{eq:SWE_lin} over the space-time domain $(x,t)\in  [0, 400000] \times [0, 4200]$. The bathymetry $B(x)$ was a simplified representation of an obstruction on the sea floor, e.g., a reef, and is given as
\begin{equation*} \label{eq:B_reef}
B(x) =
\qquad \left \{ \qquad
\begin{aligned}
 -4000,  \qquad\qquad\qquad\qquad                        & x \leq 25000, \\
 \frac{-50}{25000^2}(x-200000)(x-250000,) \quad 200000 < & x < 250000, \\
 -4000, \qquad\qquad\qquad\qquad                         & x \geq 50000,
\end{aligned}
\qquad \right \}
\end{equation*}
which is depicted in Figure \ref{fig:SWE_Reef_Bathymetry}. The rest-state $\bar{\eta} = 1$. The Dirichlet boundary conditions were again $ \mu(0, t) = \mu(D, t) = 0$, and the initial condition was $u(x,0)=(\eta_0(x),0)^{\top}$, where
\begin{equation*}
\eta_0(x) =
\qquad \left \{ \qquad
\begin{aligned}
0,    \qquad\qquad\qquad\qquad\qquad                            & x \leq 100000, \\
\frac{0.4}{25000^4}(x - 100000)^2(x - 150000)^2, \quad 100000 < & x < 150000, \\
0,    \qquad\qquad\qquad\qquad\qquad                            & x \geq 150000.
\end{aligned}
\qquad \right\}
\end{equation*}
The initial wave height $\eta_0(x)$ is presented in Figure \ref{fig:SWE_Reef_InitialWaveHeight}.

In order to define the functional $G(u;t)$, the weight function $w=(w_1(x),0)^{\top}$ was given by
\begin{equation*}
w_1(x) =
\qquad \left \{ \qquad
\begin{aligned}
0,    \qquad\qquad\qquad\qquad\qquad                         & x \leq 10000, \\
\frac{1}{7500^4}(x - 10000)^2(x - 25000)^2, \quad    10000 < & x < 15000, \\
0,    \qquad\qquad\qquad\qquad\qquad                         & x \geq 25000.
\end{aligned}
\qquad \right\}
\end{equation*}
and the non-zero component of the weight function $w_1(x)$ is depicted in Figure \ref{fig:SWE_Reef_WeightFunction}.

\begin{figure}[ht]
\centering
\subfloat[$B(x)$]{
\includegraphics[width=5.25cm]{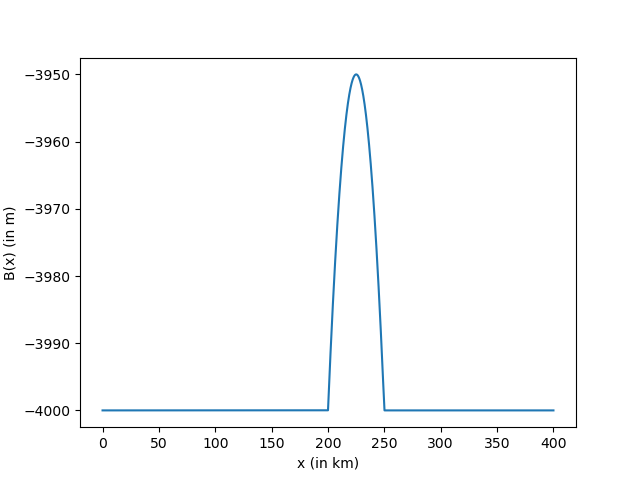}
\label{fig:SWE_Reef_Bathymetry}}
\hfill
\subfloat[$\eta_0(x)$]{
\includegraphics[width=5.25cm]{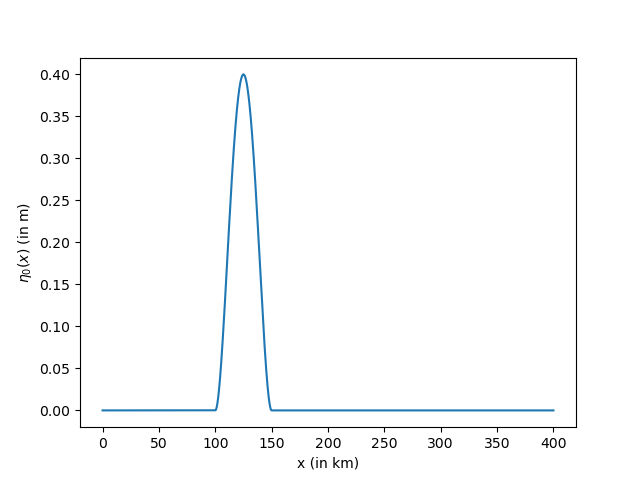}
\label{fig:SWE_Reef_InitialWaveHeight}}
\hfill
\subfloat[$w_1(x)$]{
\includegraphics[width=5.25cm]{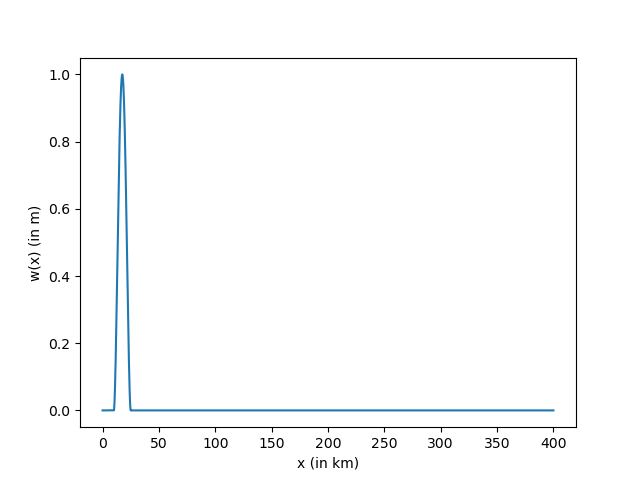}
\label{fig:SWE_Reef_WeightFunction}}
\caption{The bathymetry $B(x)$, initial wave height $\eta_0(x)$, and nonzero component of the weight function $w_1(x)$ for \S \ref{sec:SWE_Reef}.}
\end{figure}

The reference solution was computed on a mesh with $N_{\rm ref}=1280$ sub-intervals with the cG(3,3) method. The reference values of the functional $G(u; t)$ are plotted in Figure \ref{fig:Reference_QoI_SWE_Reef} where the critical value of $R=1000$ is presented as a dashed horizontal line. While this appears more complicated than Figures \ref{fig:Reference_QoI_SWE_Constant} and \ref{fig:Reference_QoI_SWE_Shelf} this is merely because the weight function is located close to a domain boundary and so $G(u;t)$ does not have time to ``recover'' and drop back to a small value before the reflected wave impinges on the critical region. The QoIs we consider are the times for the first, second and third occurrences of the event $G(u;t)=R$.


The initial conditions  for the adjoint problems  in \S \ref{sec:SWE} are
\begin{equation*}
\begin{aligned}
    \psi^{(1)}(x) &= (w_1(x),0)^{\top}, \\
    \psi^{(2)}(x) &= (0,{-}w_{1,x}(x))^{\top}.
\end{aligned}
\end{equation*}
The forward problem was solved using the cG(2,2) method (\S \ref{sec:Galerkin_Method}) on a series of increasingly fine uniform meshes. The computed times $t_c$ to the first, second and third occurrences of the event $G(u;t)=R$, errors in the QoI $e_Q$, and effectivity ratios $\rho_{\rm eff}$, are presented in Table \ref{tab:SWE_Reef}, where the effectivity ratios are all close to one.
\begin{table}[H]
    \centering
    \begin{tabular}{||c||c|c|c||c|c|c||c|c|c||}
    \hline
         & \multicolumn{3}{c||}{First event} & \multicolumn{3}{c||}{Second event} & \multicolumn{3}{c||}{Third event} \\
         \hline
         \hline
         $N$ & $t_c$ & $e_Q$ & $\rho_{\rm eff}$ & $t_c$ & $e_Q$ & $\rho_{\rm eff}$ & $t_c$ & $e_Q$ & $\rho_{\rm eff}$ \\
         \hline
         80  & 489.353 & -3.398 $\times 10^{0}$  & 1.016 & 615.540 & -9.503 $\times 10^{0}$  & 1.260 & 663.188 & -6.692 $\times 10^{0}$   & 0.934 \\
         160 & 485.870 & 8.481  $\times 10^{-2}$ & 1.000 & 607.105 & -1.069 $\times 10^{0}$  & 1.010 & 656.738 & -2.412 $\times 10^{-1}$  & 0.981 \\
         320 & 485.951 & 4.482  $\times 10^{-3}$ & 1.000 & 606.055 & -1.869 $\times 10^{-2}$ &  1.043 & 656.478 & 1.899  $\times 10^{-2}$  & 0.942 \\
         640 & 485.954 & 7.146  $\times 10^{-4}$ & 1.001 & 606.038 & -1.194 $\times 10^{-3}$ & 1.016 & 656.501 & -4.511 $\times 10^{-3}$  & 0.960 \\
         \hline
    \end{tabular}
    \caption{Computed times $t_c$ to the first, second and third occurrences of the event $G(u;t)=R$, error in the QoI $e_Q$, and effectivity ratio $\rho_{\rm eff}$, for \S \ref{sec:SWE_Reef}, using cG(2,2) to solve the forward problem and cG(4,4) to solve the adjoint problems. The first, second, and third events occur at $t_t = 485.95515$, $t_t = 606.03632$, and $t_t = 656.49655$ respectively.}
    \label{tab:SWE_Reef}
\end{table}


\section{Conclusions}
\label{sec:Conclusion}

We have developed and implemented an accurate adjoint-based \textit{a posteriori} error estimates for the time to an event, by which we mean the time for a functional of the solution to a time-dependent PDEs to achieve a threshold value. These estimates are based on the use of Taylor's Theorem on the functional defining the QoI, and multiple applications of standard adjoint-based error analysis, to achieve a computable estimate. Implementations for the 1D heat equation and the linearized shallow water equation provide examples of both parabolic and hyperbolic PDEs. The error estimates provides a basis for developing refinement strategies to reduce the error in the time to the event. A more detailed analysis is required to partition the error in to contributions arising from the independent spatial and temporal discretizations (see e.g., \cite{chaudhry2021posteriori}) and to determine the overall effect on the error from terms arising in both the numerator and denominator of \eqref{eq:error_estimate_computable}.

\section*{Acknowledgments}
J. Chaudhry’s and Z. Stevens's work is supported by the NSF-DMS 1720402.
D. Estep's and S. Tavener's work is supported by NSF-DMS 1720473.
D. Estep's work is supported by NSERC.

\bibliography{refs_new}
\bibliographystyle{plain}

\end{document}